\documentclass[11pt,twoside]{article}
\usepackage{times}
\usepackage{amsmath,amssymb,amsthm}
\usepackage{enumerate}
\usepackage{cite}
\usepackage{graphicx}
\usepackage{epstopdf}
\usepackage{xcolor}

\allowdisplaybreaks

\newcommand{\R}{\mathbb R}

\newcommand{\p}{\partial}
\newcommand{\ve}{\varepsilon}
\newcommand{\f}{\frac}

\newcommand{\al}{\alpha}

\renewcommand{\o}{\omega}

\newcommand{\ds}{\displaystyle}

\allowdisplaybreaks

\theoremstyle{plain}
\newtheorem{theorem}{Theorem}[section]
\newtheorem{proposition}{Proposition}[section]
\newtheorem{lemma}[theorem]{Lemma}

\theoremstyle{definition}

\theoremstyle{remark}
\newtheorem{remark}{Remark}[section]

\numberwithin{equation}{section}

\title{Global small data weak solutions of 2-D semilinear wave equations
with  scale-invariant  damping, III}

\author{Li Qianqian, \qquad Yin Huicheng$^{}$\footnote{Li Qianqian (\texttt{214597007@qq.com}) and Yin Huicheng
    (\texttt{huicheng@} \texttt{nju.edu.cn}, \texttt{05407@njnu.edu.cn}) are supported by the NSFC
    (No.~12331007) and by the National key research and development programs of China (No. 2020YFA0713803, No. 2024YFA1013300).}\vspace{0.5cm}\\
  \small School of Mathematical Sciences and Mathematical Institute,\\
\small Nanjing Normal University, Nanjing 210023, China.\\}
\vspace{0.5cm}

\begin{document}

\date{}

\maketitle
\thispagestyle{empty}

\begin{abstract}
For the $2$-D semilinear wave equation with  scale-invariant  damping
$\square u+\f{\mu}{t}\p_tu=|u|^p$,
where $t\ge 1$, $\mu>0$ and $p>1$,  it is conjectured that  the
global small data weak solution $u$ exists when $p>p_{s}(2+\mu)
=\f{\mu+3+\sqrt{\mu^2+14\mu+17}}{2(\mu+1)}$ for  $0<\mu\le 2$ and $p>p_f(2)=2$ for $\mu\ge 2$. In our previous papers,
the global small solution $u$ has been obtained for $p>p_{s}(2+\mu)$
and $0<\mu<2$ but $\mu\not=1$. In the present paper, by the vector field method
together with the delicate analysis on the Bessel functions, we will show the global existence of small solution $u$ for
$p>2$ and $\mu>2$. In forthcoming paper, for $\mu=1$ and $p>p_{s}(2+\mu)=p_{s}(3)=1+\sqrt 2$,
the global solution $u$ is also obtained.
Therefore, collecting our series of conclusions together with partial results from others, this open question has been solved completely.
\end{abstract}


\qquad \quad Klainerman-Sobolev inequality

\vskip 0.1 true cm


\tableofcontents

\section{Introduction}
In this paper, we continue to study  the 2-D semilinear wave equation with scale-invariant  damping
\begin{equation}\label{equ:eff1}
\left\{ \enspace
\begin{aligned}
&\square u+\f{\mu}{t}\,\p_tu=|u|^p, \\
&u(1,x)=\ve u_0(x), \quad \partial_{t} u(1,x)=\ve u_1(x),
\end{aligned}
\right.
\end{equation}
where $\mu>0$,  $p>1$, $t\ge 1$, $x=(x_1, x_2)\in\Bbb R^2$,
$\p_i=\p_{x_i}$ ($i=1, 2$), $\square=\partial_t^2-\Delta$ with $\Delta=\p_1^2+\p_2^2$, $u_k\in
C_0^{\infty}(\R^2)$ and $\operatorname{supp} u_k\in B(0, 1)$  ($k=0, 1$), $\ve>0$ is small.
The equation in \eqref{equ:eff1} is also called as the semilinear Euler-Poisson-Darboux equation (see \cite{DA-0}
and the references therein). For  more detailed physical backgrounds on the linear wave operator $\square+\f{\mu}{t}\p_t$,
one can be referred to \cite{LWY} and so on.

Indicate the Strauss index $p_s(z)=\f{z+1+\sqrt{z^2+10z-7}}{2(z-1)}$ ($z>1$)
as a positive root of the quadratic  algebraic equation
$(z-1)p^2-(z+1)p-2=0$ (see \cite{Strauss} for the Strauss index $p_s(n)=\f{n+1+\sqrt{n^2+10n-7}}{2(n-1)}$ of the semilinear wave equation
$\square u=|u|^p$), and $p_f(n)=1+\f{2}{n}$ as the Fujita index
(see \cite{Fuj} for the Fujita index $p_f(n)=1+\f{2}{n}$ of the semilinear parabolic equation
$\partial_t u-\Delta u=|u|^p$).
In terms of \cite{Rei1} and \cite{Imai}, there is an interesting open question on \eqref{equ:eff1} as follows:

{\bf Open question (A).} {\it For 2-D problem \eqref{equ:eff1},

{\bf (A1)} when $0<\mu<2$ and $p>p_s(2+\mu)$, the small solution $u$ exists globally;

{\bf (A2)} when $\mu\ge 2$ and $p>p_f(2)=2$, there exists a global solution $u$.
}

It is pointed out that the blowup results of problem \eqref{equ:eff1} have been established for
$1<p\le \max\{p_s(2+\mu), p_f(2)\}$ (see \cite{IS}, \cite{LTW}, \cite{PR-0}-\cite{PR-1}, \cite{TL2} and \cite{W1}).
On the other hand, in our previous paper \cite{LWY} and \cite{HLWY},
{\bf (A1)} has been proved except $\mu=1$. Note that {\bf (A2)} with  $\mu\ge 3$ or $\mu=2$
holds (see \cite{DA, Rei1}). Therefore, in order to solve {\bf (A2)}, we only require to
study the case of $2<\mu<3$ in \eqref{equ:eff1}.
Our main result is stated as follows.
\begin{theorem}\label{YH-1}
For $2<\mu<3$ and $p>p_{f}(2)=2$, there exists a small constant $\varepsilon_0>0$ such that when $0<\ve<\ve_0$,
\eqref{equ:eff1} admits a unique global solution $u\in C\left([1, \infty) ; H^{2}\right) \cap C^{1}\left([1, \infty) ; H^{1}\right)\cap C^{2}\left([1, \infty) ; L^{2}\right)$.
\end{theorem}

\begin{remark}\label{RE-1}
{\it In our forthcoming  paper \cite{HLWY-1},
{\bf (A1)} with $\mu=1$ and $p>p_s(2+\mu)=1+\sqrt{2}$  will be shown by applying
the Bessel function tools and deriving some suitable time-decay or spacetime-decay estimates of weak solution $u$
as in \cite{DA-0}, \cite{HLWY}  and \cite{LWY}.
Collecting the results in \cite{LWY} and \cite{HLWY} (for $0<\mu<2$ but $\mu\not=1$, $p>p_s(2+\mu)$), \cite{HLWY-1}
 (for $\mu=1$, $p>p_s(2+\mu)=1+\sqrt{2}$) together with
Theorem \ref{YH-1} in the paper (for $2<\mu<3$ and $p>2$) , \cite[Theorem 2]{Rei1}  (for $\mu=2$ and $p>2$)
and \cite[Theorem 2]{DA} (for $\mu\ge 3$ and $p>2$),
then the {\bf Open question (A)} has been solved completely.}
\end{remark}

\begin{remark}\label{RE-2}
{\it For the $n$-dimensional ($n\ge3$) semilinear wave equation with scale-invariant  damping
\begin{equation}\label{equ:eff1-2}
\left\{ \enspace
\begin{aligned}
&\square u +\f{\mu}{t}\,\p_tu=|u|^p, \\
&u(1,x)=\ve u_0(x), \quad \partial_{t} u(1,x)=\ve u_1(x),\quad x\in\Bbb R^n,
\end{aligned}
\right.
\end{equation}
where $\square=\partial_t^2-(\p_1^2+\cdot\cdot\cdot+\p_n^2)$, $\mu\in (0,1)\cup (1,2)$ and $p>p_s(n+\mu)=\f{n+\mu+1+\sqrt{(n+\mu)^2+10(n+\mu)-7}}{2(n+\mu-1)}$,
as pointed in Remark 1.12 of \cite{HLWY}, the global existence
of $u$ can be analogously shown by the methods in \cite{LWY} and \cite{HLWY}.
In addition, for $\mu=1$, $n\ge 4$ and $p>p_s(n+1)=\f{n+2+\sqrt{n^2+12n+4}}{2n}$,
the authors in \cite{HL} establish the global solution $u$ to problem \eqref{equ:eff1-2}.}
\end{remark}

\begin{remark}\label{RE-3}
{\it We emphasize that the series of methods developed in \cite{LWY}, \cite{HLWY}
and the present paper will be very helpful in showing the global solvability of
\eqref{equ:eff1-2} when $\mu>0$ and $p>\max\{p_s(n+\mu), p_f(n)\}$.}
\end{remark}

In order to prove Theorem \ref{YH-1}, let us recall the important vector field method in studying the
global small data solution problem of $n-$dimensional  ($n\ge 2$)  quasilinear wave equation
$\square \phi+\sum_{i,j,k=0}^ng_{ij}^k\p_{ij}^2\phi\p_k\phi=0$
(when $n=2,3$, the related null condition holds, namely, $\sum_{i,j,k=0}^ng_{ij}^k\xi_i\xi_j\xi_k$
$\equiv 0$ for $\xi_0=-1$ and $(\xi_1, \cdot\cdot\cdot, \xi_n)\in S^n$).
Define the vector field $Z\in\{\p, L_0, L_j, \Omega_{i j}:  1\le i<j\le n\}$,
where $\p=(\partial_0, \partial_1, \ldots, \partial_n)$ $=(\partial_t, \partial_{x_1}, \ldots, \partial_{x_n})$, $L_0=t \partial_t+x_1 \partial_1+\cdots+x_n \partial_n$,
$L_j=t\partial_j+x_j \partial_t$ and $\Omega_{i j}=x_i \partial_j-x_j \partial_i$.
Then one easily has the following commutator properties
$$[\square, \p]=0,\quad [\square, L_j]=0, \quad [\square, \Omega_{i j}]=0,\quad [\square, L_0]=2\square.$$
Based on this, it holds that for $Z^I\phi$ ($|I|\le N_0$ with $N_0\in\Bbb N$ being suitably  large),
\begin{equation}\label{YSon-0}
\square Z^I\phi+\sum_{i,j,k=0}^n\sum_{|I'|+|I''|\le|I|}C_{ij}^k\p_{ij}^2Z^{I'}\phi\p_kZ^{I''}\phi=0.
\end{equation}
Therefore, it follows from \eqref{YSon-0} that $\{Z^I\phi\}_{|I|\le N_0}$ consists of a closed nonlinear wave system
and then $\|Z^I\p \phi\|_{L_x^2(\mathbb R^n)}$ can be obtained by the standard energy estimate.
By the crucial Klainerman-Sobolev inequality
\begin{equation}\label{YSon-1}
|\p \phi(t,x)|\le \f{C}{(1+|t-r|)^{\f12}(1+t)^{\f{n-1}{2}}}\ds\sum_{|J|\le [\f{n}{2}]+1}\|Z^J\p \phi(t,x)\|_{L_x^2(\mathbb R^n)},
\end{equation}
where $r=|x|=\sqrt{x_1^2+\cdot\cdot\cdot+x_n^2}$, then the time-decay and spacetime-decay rates of $\p\phi$ are obtained.
Making use of \eqref{YSon-1} and delicate weighted energies,
a lot of interesting results on the global existence of the small data
solution $\phi$ have been established.
One can be referred to \cite{A}-\cite{A3}, \cite{Ding1}-\cite{Ding2}, \cite{H}, \cite{KS}-\cite{KP}, \cite{Li-Chen}-\cite{LX}
and \cite{Lin0}-\cite{Lin2}.

Note that for 2-D equation $\square u +\f{2}{t}\,\p_tu=|u|^p$, by the transformation $u=tv$,
then the resulting semilinear wave equation $\square v=t^{-(p-1)}|v|^p$ can be obtained. By the standard vector field method
and the energy estimate as in \cite{LX}, the authors in Theorem 2 of \cite{Rei1} show that the small data solution
$u\in C([1,\infty), H^2(\Bbb R^2))\cap C^1([1,\infty), H^1(\Bbb R^2))\cap C^2([1,\infty), L^2(\Bbb R^2))$ exists globally.
However, for the general 2-D equation $\square u +\f{\mu}{t}\,\p_tu=|u|^p$ with $2<\mu<3$, it is difficult to find
a suitable transformation of the unknown function $u$ such that a related semilinear wave equation is derived.
Moreover, the commutator $[\square+\f{\mu}{t}\,\p_t, Z]$ can not be expressed as an efficient linear combination of
$\square+\f{\mu}{t}\,\p_t$ and $Z$. Indeed, letting $w_1=\p_t u$ and $w_2=L_1 u$, then one has from $\square u +\f{\mu}{t}\,\p_tu=|u|^p$
that
\begin{equation}\label{YSon-6}
\square w_1 +\f{\mu}{t}\,\p_tw_1- \f{\mu}{t^2}w_1=\p_t(|u|^p)
\end{equation}
and
\begin{equation}\label{YSon-7}
\square w_2 +\f{\mu}{t}\,\p_tw_2-\f{\mu}{t}\p_1u- \f{\mu x_1}{t^2}w_1=L_1(|u|^p).
\end{equation}
This means that the linear parts in the nonlinear equations \eqref{YSon-6} and \eqref{YSon-7}
become more complicated than the Euler-Poisson-Darboux operator together with the Klein-Gordon type's operator.
To overcome this essential difficulty, we will take the following measures.

$\bullet$ By the Bessel function and Fourier analysis tools, through the explicit expression of
solution to the 2-D linear equation
\begin{equation}\label{YSon-2}
\square v+\f{\mu}{t}\,\p_tv=F(t,x),\quad (v,\p_tv)(1,x)=(v_0,v_1)(x),
\end{equation}
the required estimates of $v$ in some suitable norms are established.
We will utilize the following norms for $v$:
\begin{equation}\label{norm}
\|v\|_{Z, s,(p, q)}=\sum_{|\alpha| \leq s}\left\|Z^\alpha v\right\|_{(p, q)}, \quad
\|v\|_{Z, s, p}=\|v\|_{Z, s,(p, p)}, \quad
\|v\|_{Z, s, \infty}=\sum_{|\alpha| \leq s}\left\|Z^\alpha v\right\|_{\infty},
\end{equation}
where $s\in\Bbb N_0$, and $\|v\|_{(p,q)} :=\|v(r \omega) r^{\frac{1}{p}}\|_{L^p([0,+\infty); L^q(S^1))}$
for $p,q\in[1,\infty]$ and $\o=\f{x}{r}\in S^1$.
It is pointed out that the norms in \eqref{norm} and in higher space dimensions are firstly introduced in \cite{LX}
for studying the lower bounds
of the lifespan of classical solutions to the Cauchy problem for fully nonlinear wave equation
$\square u=F(u,\p u, \p^2 u)$, where the space dimensions $n\ge 3$ and the smooth nonlinearity
$F(u,\p u, \p^2 u)=O(|(u, \p u, \p^2 u)|^{1+\lambda})$ ($\lambda\in\Bbb N$).

$\bullet$ Introducing the function space $X(T)=\{u(t,x)\in X(T): \sup _{t \in[1, T]}(t^{-(\delta-1)}\|u\|_{Z, 1, 2}
+t\|\p u\|_{Z, 1, 2})<\infty\}$
and the norm $\|u\|_{X(T)}:=\sup _{t \in[1, T]}(t^{-(\delta-1)}\|u\|_{Z, 1, 2}+t\|\p u\|_{Z, 1, 2})$,
where $T>1$ is any fixed number, $\delta>0$ is a suitable constant.
On the other hand, it follows from \eqref{equ:eff1} that one may define such a nonlinear mapping
\begin{equation}\label{Yson-3}
\mathcal{N} u(t, x)=\ve\Psi_{0}(t, 1, D)u_{0}(x)
+\ve\Psi_{1}(t, 1, D)u_{1}(x)+\int_{1}^{t} \Psi_{1}(t, \tau, D)|u(\tau, x)|^{p} \mathrm{d} \tau,
\end{equation}
where the symbols of the pseudodifferential operators $\Psi_{0}(t, 1, D)$ and $\Psi_{1}(t, 1, D)$
are given in \eqref{equ:q5} and \eqref{equ:q6} of Section \ref{sec2-1} below.
Based on the key estimates for \eqref{norm}, we can show that the mapping $\mathcal{N}$ in \eqref{Yson-3}
is contractible in a closed subspace of $X(T)$. Therefore, by the fixed point principle,
problem \eqref{equ:eff1} admits a global small solution $u\in C\left([1, \infty) ; H^{2}\right) \cap C^{1}\left([1, \infty) ; H^{1}\right)\cap C^{2}\left([1, \infty) ; L^{2}\right)$.

This paper is organized as follows. In Section 2, by the first kind of Bessel function and Hankel function,
the explicit expression of the solution $v$ is given for the linear homogeneous equation $\square v +\f{\mu}{t}\,\p_tv=0$
with $(v,\p_tv)(\tau,x)=(v_0(x), v_1(x))$.
In Section 3, a series of time-decay estimates for the solution $v$ of $\square v+\f{\mu}{t}\,\p_tv=0$
are derived under the action of the vector field $Z\in\{\p, L_0, L_j, \Omega_{12}:  1\le j\le 2\}$.
In addition, in Section 4, the time-decay estimates for the linear inhomogeneous
equation $\square w +\f{\mu}{t}\,\p_tw=F$ with $(w,\p_tw)(1,x)=(0, 0)$ are established.
In Section 5, based on the estimates in Sections 3-4,
we complete the  proof of Theorem  \ref{YH-1} by using Duhamel's principle and contraction mapping principle.

\vskip 0.2 true cm

Below, we will use the following notations:

$\bullet$  For the nonnegative quantities $f$ and $g$,
$f\lesssim g$ means $f \leq C g$, where the generic constant $C > 0$ is independent of $\ve$;

$\bullet$  For the nonnegative quantities $f$ and $g$, $f \sim g$ means $g\lesssim f \lesssim g$;

$\bullet$  $\p=(\partial_t, \partial_1, \partial_2)$, $L_0=t \partial_t+x_1 \partial_1+x_2 \partial_2$,
$L_j=t\partial_j+x_j \partial_t$ ($j=1,2$) and $\Omega_{12}=x_1 \partial_2-x_2 \partial_1$;

$\bullet$  $Z\in\{\p, L_0, L_j, \Omega_{12}:  1\le j\le 2\}$;

$\bullet$ For $p,q\in[1,\infty]$, $\|v\|_{(p,q)} :=\|v(r \omega) r^{\frac{1}{p}}\|_{L^p([0,+\infty); L^q(S^1))}$
with $\o=\f{x}{r}\in S^1$;

$\bullet$ For $s\in\Bbb N_0$ and  $p,q\in[1,\infty]$,
\begin{equation*}\label{norm-1}
\|v\|_{Z, s,(p, q)}=\sum_{|\alpha| \leq s}\left\|Z^\alpha v\right\|_{(p, q)}, \quad
\|v\|_{Z, s, p}=\|v\|_{Z, s,(p, p)}, \quad
\|v\|_{Z, s, \infty}=\sum_{|\alpha| \leq s}\left\|Z^\alpha v\right\|_{\infty},
\end{equation*}

\quad where $\|v\|_{(p,q)} :=\|v(r \omega) r^{\frac{1}{p}}\|_{L^p([0,+\infty); L^q(S^1))}$
and $\o=\f{x}{r}\in S^1$;

$\bullet$ For $\xi\in\Bbb R^2$ and $t\ge 1$, define
\begin{equation*}\label{equ:q19}
A_1=\{\xi: |\xi|\ge 1\}, \quad A_2=\{\xi: |\xi| \leq 1 \leq t|\xi|\}, \quad A_3=\{\xi: t|\xi| \leq 1\}.
\end{equation*}

\section{The expression of solution to 2-D equation $\square v+\f{\mu}{t}\,\p_tv=0$}\label{sec2-1}

In this section, we derive the expression of the solution $v$ to 2-D problem
\begin{equation}\label{equ:q1}
\left\{ \enspace
\begin{aligned}
&\square v+\f{\mu}{t}\,\p_tv=0, &&
t\geq \tau \geq1,\\
&v(\tau,x)=v_0(x), \quad \partial_{t} v(\tau,x)=v_1(x), &&x\in\R^2,
\end{aligned}
\right.
\end{equation}
where $2<\mu<3$ and $v_i(x)\in C_0^{\infty}(\R^2)$ $(i=0, 1)$.

Taking the Fourier transformation of $v(t,x)$ with respect to the variable $x$, then
one has from \eqref{equ:q1} that
\begin{equation}\label{equ:q2}
\left\{ \enspace
\begin{aligned}
&\partial_t^2 \hat{v} +|\xi|^{2} \hat{v} +\f{\mu}{t}\,\p_t\hat{v}=0, &&
t\geq \tau \geq 1,\\
&\hat{v}(\tau,\xi)=\hat{v}_0(\xi), \quad \partial_{t} \hat{v}(\tau,\xi)=\hat{v}_1(\xi), &&\xi\in\R^2.
\end{aligned}
\right.
\end{equation}
It follows from   \cite[Theorem 2.1]{Wirth-3} that the solution $\hat{v}(t, \xi)$ of
\eqref{equ:q2} can be expressed as
\begin{equation}\label{equ:q4}
\hat{v}(t, \xi)=\Psi_{0}(t, \tau, \xi) \hat{v}_{0}(\xi)+\Psi_{1}(t, \tau, \xi) \hat{v}_{1}(\xi),
\end{equation}
where the multipliers $\Psi_{j}(t, \tau, \xi)(j=0,1)$ and its time derivatives are given by
\begin{equation}\label{equ:q5}
\Psi_{0}(t, \tau, \xi)=\frac{i \pi}{4}|\xi| \frac{t^{\rho}}{\tau^{\rho-1}}\left|\begin{array}{ll}
H_{\rho}^{+}(t|\xi|) & H_{\rho-1}^{+}(\tau|\xi|) \\
H_{\rho}^{-}(t|\xi|) & H_{\rho-1}^{-}(\tau|\xi|)
\end{array}\right|,
\end{equation}
\begin{equation}\label{equ:q6}
\Psi_{1}(t, \tau, \xi)=-\frac{i \pi}{4} \frac{t^{\rho}}{\tau^{\rho-1}}\left|\begin{array}{ll}
H_{\rho}^{+}(t|\xi|) & H_{\rho}^{+}(\tau|\xi|) \\
H_{\rho}^{-}(t|\xi|) & H_{\rho}^{-}(\tau|\xi|)
\end{array}\right|,
\end{equation}
\begin{equation}\label{equ:q9}
\partial_{t}^{k} \Psi_{j}(t, \tau, \xi)=(-1)^{j} \frac{i \pi}{4}|\xi|^{1+k-j} \frac{t^{\rho}}{\tau^{\rho-1}}\left|\begin{array}{ll}
H_{\rho-k}^{+}(t|\xi|) & H_{\rho-1+j}^{+}(\tau|\xi|)  \\
H_{\rho-k}^{-}(t|\xi|) & H_{\rho-1+j}^{-}(\tau|\xi|)
\end{array}\right|,\quad  k,j=0,1,
\end{equation}
where $\rho=-\f{\mu-1}{2}\in (-1, -\f12)$, $H_{\nu}^{\pm}$ are the Hankel functions of order $\nu$ ($\nu\not\in\Bbb Z$),
which can be expressed as (see \cite[Section 10]{OLBC})
\begin{equation}\label{equ:q10}
\begin{aligned}
&H_{\nu}^{+}(z)=i \csc (\nu \pi)\left(e^{-\nu \pi i} J_{\nu}(z)-J_{-\nu}(z)\right),\\
&H_{\nu}^{-}(z)=i \csc (\nu\pi)\left(J_{-\nu}(z)-e^{\nu\pi i} J_{\nu}(z)\right)
\end{aligned}
\end{equation}
with the first kind of Bessel function $J_{\gamma}(z)=(\f{z}{2})^{\gamma}\sum_{k=0}^{\infty}(-1)^k\f{z^{2k}}{4^kk!\Gamma(\gamma+k+1)}$
for $\gamma\in \mathbb{C} \setminus \Bbb{Z}$.

Therefore,
\begin{equation}\label{equ:q11}
\partial_{t}^{k} \Psi_{j}(t, \tau, \xi)=\frac{\pi}{2} \csc  (\rho \pi)|\xi|^{1+k-j} \frac{t^{\rho}}{\tau^{\rho-1}}
 \left|\begin{array}{cc}
J_{-(\rho-1+j)}(\tau|\xi|) & J_{-(\rho-k)}(t|\xi|) \\
(-1)^{1+k-j} J_{\rho-1+j}(\tau|\xi|) & J_{\rho-k}(t|\xi|)
\end{array}\right|,\quad  k, j=0,1.
\end{equation}

Based on \eqref{equ:q11}, we have
\begin{lemma}\label{lem6}
It holds that for $\rho=-\f{\mu-1}{2}\in (-1, -\f12)$,
\begin{equation}\label{l-1}
\begin{aligned}
\p_t^2\Psi_{0}(t, \tau, \xi)& =\frac{i \pi}{4}|\xi|^3 \frac{t^{\rho}}{\tau^{\rho-1}}\left|\begin{array}{ll}
H_{\rho-2}^{+}(t|\xi|) & H_{\rho-1}^{+}(\tau|\xi|) \\
H_{\rho-2}^{-}(t|\xi|) & H_{\rho-1}^{-}(\tau|\xi|)
\end{array}\right|+ t^{-1}\p_t\Psi_{0}(t, \tau, \xi),
\end{aligned}
\end{equation}
\begin{equation}\label{l-2}
\begin{aligned}
\p_t^2\Psi_{1}(t, \tau, \xi)& =-\frac{i \pi}{4}|\xi|^2 \frac{t^{\rho}}{\tau^{\rho-1}}\left|\begin{array}{ll}
H_{\rho-2}^{+}(t|\xi|) & H_{\rho}^{+}(\tau|\xi|) \\
H_{\rho-2}^{-}(t|\xi|) & H_{\rho}^{-}(\tau|\xi|)
\end{array}\right|+ t^{-1}\p_t\Psi_{1}(t, \tau, \xi).
\end{aligned}
\end{equation}
Meanwhile, one also has
\begin{equation}\label{l-3}
\begin{aligned}
\p_t^2\Psi_{0}(t, \tau, \xi)
& =\frac{\pi}{2}|\xi|^3 \frac{t^{\rho}}{\tau^{\rho-1}}\csc(\rho\pi)\left|\begin{array}{ll}
J_{-(\rho-1)}(\tau|\xi|) & J_{-(\rho-2)}(t|\xi|) \\
-J_{\rho-1}(\tau|\xi|) & J_{\rho-2}(t|\xi|)
\end{array}\right|+ t^{-1}\p_t\Psi_{0}(t, \tau, \xi),
\end{aligned}
\end{equation}
\begin{equation}\label{l-4}
\begin{aligned}
\p_t^2\Psi_{1}(t, \tau, \xi)
& =\frac{ \pi}{2}|\xi|^2 \frac{t^{\rho}}{\tau^{\rho-1}}\csc(\rho\pi)\left|\begin{array}{ll}
J_{-\rho}(\tau|\xi|) & J_{-(\rho-2)}(t|\xi|) \\
J_{\rho}(\tau|\xi|) & J_{\rho-2}(t|\xi|)
\end{array}\right|+ t^{-1}\p_t\Psi_{1}(t, \tau, \xi).
\end{aligned}
\end{equation}
\end{lemma}
\begin{proof}
Motivated by \cite[Corollary 2.2]{Wirth-3}, setting $r=t|\xi|$,
then it follows from \eqref{equ:q2} and the assumption of $\hat{v}=r^{\rho} y(r)$ that $y(r)$ satisfies
\begin{equation}\label{equ:q3}
r^{2} y^{\prime \prime}(r)+r y^{\prime}(r)+\left(r^{2}-\rho^{2}\right) y(r)=0.
\end{equation}
Hence, by \eqref{equ:q3}, \eqref{equ:q2} admits two independent solutions
$\text {$\hat{v}_{+}(r)=r^{\rho} H_{\rho}^{+}(r)$ and $\hat{v}_{-}(r)=r^{\rho} H_{\rho}^{-}(r)$}$.
Therefore,
\begin{equation}\label{l-5}
\begin{aligned}
& \Psi_{0}(t, \tau, \xi)=C_{01}(\tau, \xi) \hat{v}_{+}(t, \xi)+C_{02}(\tau, \xi) \hat{v}_{-}(t, \xi), \\
& \Psi_{1}(t, \tau, \xi)=C_{11}(\tau, \xi) \hat{v}_{+}(t, \xi)+C_{12}(\tau, \xi) \hat{v}_{-}(t, \xi),
\end{aligned}
\end{equation}
where
\begin{equation}\label{l-6}
\hat{v}_{+}(t, \xi)=t^{\rho}|\xi|^{\rho} H_{\rho}^{+}(t|\xi|), \quad \hat{v}_{-}(t, \xi)=t^{\rho}|\xi|^{\rho} H_{\rho}^{-}(t|\xi|).
\end{equation}
In addition, it follows from (16)-(19) in \cite{Wirth-3} that
\begin{equation}\label{l-7}
\begin{aligned}
& \left(\begin{array}{cc}
C_{01} & C_{11} \\
C_{02} & C_{12}
\end{array}\right)=\frac{1}{\left|\begin{array}{cc}
\hat{v}_{+}(\tau, \xi) & \hat{v}_{-}(\tau, \xi) \\
\hat{\partial}_{t} \hat{v}_{+}(\tau, \xi) & \partial_{t} \hat{v}_{-}(\tau, \xi)
\end{array}\right|}\left(\begin{array}{cc}
\partial_{t} \hat{v}_{-}(\tau, \xi) & -\hat{v}_{-}(\tau, \xi) \\
-\partial_{t} \hat{v}_{+}(\tau, \xi) & \hat{v}_{+}(\tau, \xi)
\end{array}\right) \\
= & \frac{i \pi}{4 \tau^{2 \rho-1}|\xi|^{2 \rho}}\left(\begin{array}{cc}
\tau^{\rho}|\xi|^{\rho+1} H_{\rho-1}^{-}(\tau|\xi|) & -\tau^{\rho}|\xi|^{\rho} H_{\rho}^{-}(\tau|\xi|) \\
-\tau^{\rho}|\xi|^{\rho+1} H_{\rho-1}^{+}(\tau|\xi|) & \tau^{\rho}|\xi|^{\rho} H_{\rho}^{+}(\tau|\xi|)
\end{array}\right).
\end{aligned}
\end{equation}
By $\nu H_{\nu}^{\pm}(z)+z(H_{\nu}^{\pm})^{\prime}(z)=zH_{\nu-1}^{\pm}(z)$ and \eqref{l-6}, we arrive at
\begin{equation}\label{l-8}
\begin{aligned}
& \partial_{t} \hat{v}_{+}(t, \xi)=t^{\rho}|\xi|^{\rho+1} H_{\rho-1}^{+}(t|\xi|),~ \partial_{t} \hat{v}_{-}(t, \xi)=t^{\rho}|\xi|^{\rho+1} H_{\rho-1}^{-}(t|\xi|),\\
& \partial_{t}^2 \hat{v}_{+}(t, \xi)=t^{\rho}|\xi|^{\rho+2} H_{\rho-2}^{+}(t|\xi|)+t^{\rho-1}|\xi|^{\rho+1} H_{\rho-1}^{+}(t|\xi|),\\ &\partial_{t}^2 \hat{v}_{-}(t, \xi)=t^{\rho}|\xi|^{\rho+2} H_{\rho-2}^{-}(t|\xi|)+t^{\rho-1}|\xi|^{\rho+1} H_{\rho-1}^{-}(t|\xi|).
\end{aligned}
\end{equation}
Substituting \eqref{l-7} and \eqref{l-8} into \eqref{l-5}  yields
\begin{equation}\label{l-9}
\begin{aligned}
\partial_{t}^2 \Psi_{0}(t, \tau, \xi)
&= \frac{i \pi}{4}|\xi|^3 \frac{t^{\rho}}{\tau^{\rho-1}}\left|\begin{array}{ll}
H_{\rho-2}^{+}(t|\xi|) & H_{\rho-1}^{+}(\tau|\xi|) \\
H_{\rho-2}^{-}(t|\xi|) & H_{\rho-1}^{-}(\tau|\xi|)
\end{array}\right|+\frac{i \pi}{4}|\xi|^2 \frac{t^{\rho-1}}{\tau^{\rho-1}}\left|\begin{array}{ll}
H_{\rho-1}^{+}(t|\xi|) & H_{\rho-1}^{+}(\tau|\xi|) \\
H_{\rho-1}^{-}(t|\xi|) & H_{\rho-1}^{-}(\tau|\xi|)
\end{array}\right|\\
& =\frac{i \pi}{4}|\xi|^3 \frac{t^{\rho}}{\tau^{\rho-1}}\left|\begin{array}{ll}
H_{\rho-2}^{+}(t|\xi|) & H_{\rho-1}^{+}(\tau|\xi|) \\
H_{\rho-2}^{-}(t|\xi|) & H_{\rho-1}^{-}(\tau|\xi|)
\end{array}\right|+ t^{-1}\p_t\Psi_{0}(t, \tau, \xi)
\end{aligned}
\end{equation}
and
\begin{equation}\label{l-10}
\begin{aligned}
\partial_{t}^2 \Psi_{1}(t, \tau, \xi)
& =-\frac{i \pi}{4}|\xi|^2 \frac{t^{\rho}}{\tau^{\rho-1}}\left|\begin{array}{ll}
H_{\rho-2}^{+}(t|\xi|) & H_{\rho}^{+}(\tau|\xi|) \\
H_{\rho-2}^{-}(t|\xi|) & H_{\rho}^{-}(\tau|\xi|)
\end{array}\right|+ t^{-1}\p_t\Psi_{1}(t, \tau, \xi).
\end{aligned}
\end{equation}
Since $\rho$ is not an integer, one then has that from \eqref{equ:q10} and \eqref{l-9},
\begin{equation} \label{l-11}
\begin{aligned}
& \frac{i \pi}{4}|\xi|^3 \frac{t^{\rho}}{\tau^{\rho-1}}\left|\begin{array}{ll}
H_{\rho-2}^{+}(t|\xi|) & H_{\rho-1}^{+}(\tau|\xi|) \\
H_{\rho-2}^{-}(t|\xi|) & H_{\rho-1}^{-}(\tau|\xi|)
\end{array}\right|\\
& =\frac{i \pi}{4}|\xi|^3 \frac{t^{\rho}}{\tau^{\rho-1}} \csc ^{2}(\rho \pi)\left(e^{-(\rho-2) \pi i} J_{-(\rho-1)}(\tau|\xi|) J_{\rho-2}(t|\xi|)-J_{-(\rho-1)}(\tau|\xi|) J_{-(\rho-2)}(t|\xi|)\right. \\
& \quad\left.-e^{\pi i} J_{\rho-1}(\tau|\xi|) J_{\rho-2}(t|\xi|)+e^{(\rho-1) \pi i} J_{\rho-1}(\tau|\xi|) J_{-(\rho-2)}(t|\xi|)\right) \\
&\quad -\frac{i \pi}{4}|\xi|^3 \frac{t^{\rho}}{\tau^{\rho-1}} \csc ^{2}(\rho \pi)\left(e^{-(\rho-1) \pi i} J_{\rho-1}(\tau|\xi|) J_{-(\rho-2)}(t|\xi|)-e^{-\pi i} J_{\rho-1}(\tau|\xi|) J_{\rho-2}(t|\xi|)\right. \\
& \quad\left.-J_{-(\rho-1)}(\tau|\xi|) J_{-(\rho-2)}(t|\xi|)+e^{(\rho-2) \pi i} J_{-(\rho-1)}(\tau|\xi|) J_{\rho-2}(t|\xi|)\right) \\
& =\frac{i\pi}{4}|\xi|^3 \frac{t^{\rho}}{\tau^{\rho-1}} \csc ^2 (\rho \pi)\left[-2i\sin(\rho\pi)\cdot\left(J_{\rho-1}(\tau|\xi|)J_{\rho-2}(t|\xi|)+J_{\rho-1}(\tau|\xi|)J_{-(\rho-2)}(t|\xi|)\right)\right]\\
& =\frac{\pi}{2}|\xi|^3 \frac{t^{\rho}}{\tau^{\rho-1}} \csc  (\rho \pi)\left|\begin{array}{cc}
J_{-(\rho-1)}(\tau|\xi|) & J_{-(\rho-2)}(t|\xi|) \\
-J_{\rho-1}(\tau|\xi|) & J_{\rho-2}(t|\xi|)
\end{array}\right|.
\end{aligned}
\end{equation}
Similarly,
\begin{equation}\label{l-12}
-\frac{i \pi}{4}|\xi|^2 \frac{t^{\rho}}{\tau^{\rho-1}}\left|\begin{array}{ll}
H_{\rho-2}^{+}(t|\xi|) & H_{\rho}^{+}(\tau|\xi|) \\
H_{\rho-2}^{-}(t|\xi|) & H_{\rho}^{-}(\tau|\xi|)
\end{array}\right|=\frac{ \pi}{2}|\xi|^2 \frac{t^{\rho}}{\tau^{\rho-1}}\csc(\rho\pi)\left|\begin{array}{ll}
J_{-\rho}(\tau|\xi|) & J_{-(\rho-2)}(t|\xi|) \\
J_{\rho}(\tau|\xi|) & J_{\rho-2}(t|\xi|)
\end{array}\right|.
\end{equation}
Therefore, \eqref{l-3} and \eqref{l-4} follow from \eqref{l-11} and \eqref{l-12}, respectively.
\end{proof}

Let us  recall some useful results on the asymptotic behaviors of $J_{\nu}(z)$ and $H_{\nu}^{\pm}(z)$ ($z>0$)
for sufficiently large $z$ or small $z$. For details,
one can be refered to \cite[Section 10]{OLBC} or \cite[\S3.13, \S3.52, \S10.6 and \S7.2]{Wat}.

\begin{lemma}\label{lem1}
It holds that

(i) for large $z\geq K > 0$,
\begin{equation}\label{equ:q12}
|H_{\nu}^{\pm}(z)|\lesssim z^{-\f{1}{2}};
\end{equation}

for small $z$ with $0<z\leq c<1$,
\begin{equation}\label{equ:q13}
\left|H_{\nu}^{ \pm}(z)\right| \lesssim \begin{cases}z^{-|\nu|}, & \nu \neq 0, \\
-\log z, & \nu=0.
\end{cases}
\end{equation}

(ii) for large $z\geq K > 0$,
\begin{equation}\label{equ:q14}
|J_{\nu}(z)|\lesssim z^{-\f{1}{2}};
\end{equation}

for small $z$ with $0<z\leq c<1$,
\begin{equation}\label{equ:q15}
|J_{\nu}(z)|\lesssim z^{\nu}.
\end{equation}
\end{lemma}

\section{Estimates of solution to 2-D homogeneous equation $\square v+\f{\mu}{t}\,\p_tv=0$}\label{sec3}

Let $v$ solve
\begin{equation}\label{equ:q16}
\left\{ \enspace
\begin{aligned}
&\partial_t^2 v-\Delta v +\f{\mu}{t}\,\p_tv=0, &&
t\geq 1,\\
&v(1,x)=v_0(x), \quad \partial_{t} v(1,x)=v_1(x), &&x\in\R^2,
\end{aligned}
\right.
\end{equation}
where $2<\mu<3$ and $v_i(x)\in C_0^{\infty}(B(0, 1))$ $(i=0, 1)$.
\begin{lemma}\label{lem2}
For any $\varepsilon_1\in(0, 1)$, there exists a constant $\delta(\varepsilon_1)=\f{2\ve_1}{1+\varepsilon_1}>0$ such that
\begin{equation}\label{equ:q17}
\|v(t, \cdot)\|_{Z, 1, 2} \lesssim t^{-\f{\mu}{2}}\left\|v_0\right\|_{Z, 1, 2}
+t^{\delta(\varepsilon_1)-1}\left\|v_1\right\|_{Z, 1,(1+\varepsilon_1, 2)}
+t^{\delta(\varepsilon_1)-1}\left\|v_0+v_1\right\|_{Z, 1,(1+\varepsilon_1, 2)}.
\end{equation}
\end{lemma}
\begin{proof}
By the definition of $\|v\|_{Z, s,(p, q)}$ with $s=1$ and $p=q=2$,
we firstly establish
\begin{equation}\label{equ:q18}
\|v(t, \cdot)\|_{L^2(\R^2)} \lesssim t^{-\f{\mu}{2}}\left\|v_0\right\|_{L^2(\R^2)}+t^{\delta(\varepsilon_1)-1}\left\|v_1\right\|_{(1+\varepsilon_1, 2)}+t^{\delta(\varepsilon_1)-1}\|v_0+v_1\|_{(1+\varepsilon_1, 2)}.
\end{equation}
Note that
$\hat{v}(t, \xi)=\Psi_{0}(t, 1, \xi) \hat{v}_{0}(\xi)+\Psi_{1}(t, 1, \xi) \hat{v}_{1}(\xi)$.
In order to use the asymptotic behaviors of Bessel functions to estimate $v(t, x)$,
we will decompose the frequency $\xi\in\Bbb R^2$  into the following three  zones
\begin{equation}\label{equ:q19}
A_1=\{\xi: |\xi|\ge 1\}, \quad A_2=\{\xi: |\xi| \leq 1 \leq t|\xi|\}, \quad A_3=\{\xi: t|\xi| \leq 1\}.
\end{equation}

\vskip 0.2 true cm
{\bf Part 1. The analysis in zone $A_1$}
\vskip 0.2 true cm

By \eqref{equ:q9} with $\rho=-\f{\mu-1}{2}$ and \eqref{equ:q12}, one can obtain
\begin{equation}\label{equ:q20}
\begin{aligned}
& \left|\Psi_0(t, 1,\xi)\right| \lesssim|\xi| t^\rho|\xi|^{-\frac{1}{2}}(t|\xi|)^{-\frac{1}{2}}=t^{-\frac{\mu}{2}}, \\
& \left|\Psi_1(t, 1, \xi)\right|  \lesssim t^\rho|\xi|^{-\frac{1}{2}}(t|\xi|)^{-\frac{1}{2}}=t^{-\frac{\mu}{2}}|\xi|^{-1}.
\end{aligned}
\end{equation}
Then
\begin{equation}\label{equ:q21}
\begin{aligned}
\left\|\hat{v}(t, \cdot)\right\|_{L^2(A_1)} & \leq\left\|  \Psi_{0}(t, 1, \xi)\right\|_{L^{\infty}}\left\|\hat{v}_{0}\right\|_{L^{2}}+\left\||\Psi_{1 }(t, 1, \xi)|\hat{v}_{1}\right\|_{L^{2}} \\
& \lesssim t^{-\frac{\mu}{2}}\left\|v_{0}\right\|_{L^{2}}+t^{-\f{\mu}{2}} \|(1+|\xi|^2)^{-\f{1}{2}}\hat{v}_{1}\|_{L^{2}}
:= t^{-\frac{\mu}{2}}\left\|v_{0}\right\|_{L^{2}}+ I.
\end{aligned}
\end{equation}
Next we treat the term $I$ in \eqref{equ:q21}.  By setting $x=ty$ and $\xi=\f{\eta}{t}$, one has that for $2<\mu<3$ and $t\geq1$,
\begin{equation}\label{equ:q22}
\begin{aligned}
I & =t^{-\f{\mu}{2}} \left(\int_{R^2} \frac{1}{1+|\xi|^2}\big(\int_{R^2} e^{-i x \cdot \xi} v_1(x) \mathrm{d} x\big)^2 \mathrm{d}\xi\right)^{\frac{1}{2}}\\
& =t^{-\f{\mu}{2}+2}\left(\int_{R^2} \frac{1}{t^2+|\eta|^2}\big(\int_{R^2} e^{-i y \cdot \eta} v_1(ty) \mathrm{d} y\big)^2 \mathrm{d} \eta\right)^{\frac{1}{2}}\\
& \leq t^{-\f{\mu}{2}+2}\left(\int_{R^2} \frac{1}{1+|\eta|^2}\big(\int_{R^2} e^{-i y \cdot \eta} v_1(ty) \mathrm{d} y\big)^2 \mathrm{d} \eta\right)^{\frac{1}{2}}\\
& \lesssim t\|v_1(ty)\|_{H^{-1}(\R^2)}.
\end{aligned}
\end{equation}
Due to $\|V_1\|_{H^{-1}}=\sup _{v \in H^1, v \neq 0} \frac{\left|\int_{\mathbb{R}^2} V_1(y) v(y) \mathrm{d} y\right|}{\|v\|_{H^1}}$
for $V_1(y)=v_1(ty)$, then it follows from H\"{o}lder's inequality and Sobolev imbedding theorem that
\begin{equation}\label{equ:q23}
\|V_1 v\|_{L^1} \lesssim\|V_1\|_{(q, 2)}\|v\|_{\left(q^{\prime}, 2\right)} \lesssim\|V_1\|_{(q, 2)}\|v\|_{q^{\prime}} \lesssim\|V_1\|_{(q, 2)}\|v\|_{H^1}\lesssim t^{-\f{2}{q}}\|v_1\|_{(q, 2)}\|v\|_{H^1},
\end{equation}
where $q=1+\varepsilon_1$ with any fixed number $\varepsilon_1\in(0,1)$ and $\f{1}{q}+\f{1}{q'}=1$. Hence,
$$
\|V_1\|_{H^{-1}}\lesssim t^{-\f{2}{1+\varepsilon_1}}\|v_1\|_{(1+\varepsilon_1, 2)}.
$$
This, together with \eqref{equ:q22}, yields
$$
I\lesssim t^{1-\f{2}{1+\varepsilon_1}}\|v_1\|_{(1+\varepsilon_1, 2)}
$$
and
\begin{equation}\label{equ:q29}
\left\|\hat{v}(t, \cdot)\right\|_{L^2(A_1)}\lesssim t^{-\frac{\mu}{2}}\left\|v_{0}\right\|_{L^{2}}+ t^{1-\f{2}{1+\varepsilon_1}}\|v_1\|_{(1+\varepsilon_1, 2)}.
\end{equation}

\vskip 0.2 true cm
{\bf Part 2. The analysis in zone $A_2$}
\vskip 0.2 true cm

For $\rho=-\f{\mu-1}{2}\in(-1,-\f{1}{2})$, it follows from \eqref{equ:q11}, \eqref{equ:q14} and \eqref{equ:q15} that
\begin{equation}\label{equ:q28}
\begin{aligned}
& \left|\Psi_0(t, 1,\xi)\right| \lesssim|\xi| t^\rho|\xi|^{\rho-1}(t|\xi|)^{-\frac{1}{2}}=t^{-\frac{\mu}{2}}|\xi|^{-\frac{\mu}{2}} ,\\
& \left|\Psi_1(t, 1, \xi)\right|  \lesssim t^\rho|\xi|^{\rho}(t|\xi|)^{-\frac{1}{2}}=t^{-\frac{\mu}{2}}|\xi|^{-\frac{\mu}{2}}.
\end{aligned}
\end{equation}
Due to $t|\xi|\geq 1$, then one can utilize
$(t|\xi|)^{-\f{\mu}{2}}\lesssim (1+t|\xi|)^{-1}$,
$x=ty$ and $t\xi=\eta$ to obtain
\begin{equation}\label{equ:q25}
\begin{aligned}
\left\|\hat{v}(t, \cdot)\right\|_{L^2(A_2)} & \leq\left\|(1+t|\xi|)^{-1}(\hat{v}_{0}+\hat{v}_{1})\right\|_{L^{2}} \\
& =t\left(\int_{R^2} (\frac{1}{1+|\eta|})^2\big(\int_{R^2} e^{-i y \cdot \eta} (v_0+v_1)(ty) \mathrm{d} y\big)^2 \mathrm{d} \eta\right)^{\frac{1}{2}}\\
& \lesssim t\|(v_0+v_1)(ty)\|_{H^{-1}(\R^2)}.
\end{aligned}
\end{equation}
Analogously to the treatment of \eqref{equ:q23}, we then have
\begin{equation}\label{equ:q30}
\left\|\hat{v}(t, \cdot)\right\|_{L^2(A_2)}\lesssim t^{1-\f{2}{1+\varepsilon_1}}\|v_0+v_1\|_{(1+\varepsilon_1, 2)}.
\end{equation}

\vskip 0.2 true cm
{\bf Part 3. The analysis in zone $A_3$}
\vskip 0.2 true cm

As in the zone $A_2$, one has
\begin{equation}\label{equ:q26}
\begin{aligned}
& \left|\Psi_0(t, 1,\xi)\right| \lesssim|\xi| t^\rho|\xi|^{\rho-1}(t|\xi|)^{-\rho}=1, \\
& \left|\Psi_1(t, 1, \xi)\right|  \lesssim t^\rho|\xi|^{\rho}(t|\xi|)^{-\rho}=1.
\end{aligned}
\end{equation}
Then for $|\xi|\leq t|\xi| \leq 1$, as in \eqref{equ:q25}, we arrive at
\begin{equation}\label{equ:q27}
\begin{aligned}
\left\|\hat{v}(t, \cdot)\right\|_{L^2(A_3)} & \leq\left\|(1+t|\xi|)\times(1+t|\xi|)^{-1}(\hat{v}_{0}+\hat{v}_{1})\right\|_{L^{2}} \\
&\lesssim\left\|(1+t|\xi|)^{-1}(\hat{v}_{0}+\hat{v}_{1})\right\|_{L^{2}}\\
& \lesssim t\|(v_0+v_1)(ty)\|_{H^{-1}(\R^2)}\\
&\lesssim t^{1-\f{2}{1+\varepsilon_1}}\|v_0+v_1\|_{(1+\varepsilon_1, 2)}.
\end{aligned}
\end{equation}
Collecting \eqref{equ:q29}, \eqref{equ:q30}, \eqref{equ:q27} and Parseval's equality yields \eqref{equ:q18}.

Next we prove \eqref{equ:q17}. By the definition of $\|v\|_{Z, 1,2}$,
the proof procedure will be divided into the following four cases.

\vskip 0.2 true cm

{\bf Case 1.  The treatment  of $\|\p v(t, \cdot)\|_{L^2(\R^2)}$ }

\vskip 0.2 true cm

In this case, we shall establish the following estimate
\begin{equation}\label{equ:q31}
\|\p v(t, \cdot)\|_{L^2(\R^2)} \lesssim t^{-\f{\mu}{2}}\left\|v_0\right\|_{Z, 1, 2}+t^{\delta(\varepsilon_1)-1}\left\|v_1\right\|_{Z, 1,(1+\varepsilon_1, 2)}+t^{\delta(\varepsilon_1)-1}\left\|v_0+v_1\right\|_{Z, 1,(1+\varepsilon_1, 2)}.
\end{equation}
To prove \eqref{equ:q31}, we still
decompose the space $\R_{\xi}^2$ into three zones as in \eqref{equ:q19}.

\vskip 0.2 true cm

{\bf Case 1.1.  The analysis  in zone $A_1$}

\vskip 0.2 true cm

It follows from \eqref{equ:q9} and  \eqref{equ:q12} that
\begin{equation}\label{equ:q32}
\begin{aligned}
& \left|\p_t\Psi_0(t, 1,\xi)\right| \lesssim|\xi|^2 t^\rho|\xi|^{-\frac{1}{2}}(t|\xi|)^{-\frac{1}{2}}=t^{-\frac{\mu}{2}}|\xi|, \\
& \left|\p_t\Psi_1(t, 1, \xi)\right|  \lesssim |\xi|t^\rho|\xi|^{-\frac{1}{2}}(t|\xi|)^{-\frac{1}{2}}=t^{-\frac{\mu}{2}}.
\end{aligned}
\end{equation}
Then for $|\xi| \geq 1$ and $2<\mu<3$, as in \eqref{equ:q25}, one has
\begin{equation}\label{equ:q33}
\begin{aligned}
\left\|\p_t\hat{v}(t, \cdot)\right\|_{L^2(A_1)}
& \lesssim t^{-\frac{\mu}{2}}\left\||\xi|v_{0}\right\|_{L^{2}}+t^{-\f{\mu}{2}} \||\xi|^{-1}|\xi|\hat{v}_{1}\|_{L^{2}}\\
& \lesssim t^{-\frac{\mu}{2}}\left\|v_{0}\right\|_{Z, 1, 2}+t^{-\f{\mu}{2}+1} \|(1+t|\xi|)^{-1}|\xi|\hat{v}_{1}\|_{L^{2}}\\
& \lesssim t^{-\frac{\mu}{2}}\left\|v_{0}\right\|_{Z, 1, 2}+  t^{-\f{\mu}{2}+2}\|\nabla v_1(ty)\|_{H^{-1}}\\
&\lesssim t^{-\frac{\mu}{2}}\left\|v_{0}\right\|_{Z, 1, 2}+  t^{1-\f{2}{1+\varepsilon_1}}\|v_1\|_{Z, 1, (1+\varepsilon_1,2)}.
\end{aligned}
\end{equation}
Analogously to the treatment of $\left\|\p_t\hat{v}(t, \cdot)\right\|_{L^2(A_1)}$, we have that for $j=1, 2$,
\begin{equation}\label{equ:q34}
\left\||\xi_j|\hat{v}(t, \cdot)\right\|_{L^2(A_1)}
\lesssim t^{-\frac{\mu}{2}}\left\|v_{0}\right\|_{Z, 1, 2}+  t^{1-\f{2}{1+\varepsilon_1}}\|v_1\|_{Z, 1, (1+\varepsilon_1,2)}.
\end{equation}
This, together with \eqref{equ:q33}, yields
\begin{equation}\label{equ:q35}
\|\widehat{\p v}\|_{L^2(A_1)}\leq \left\|\p_t\hat{v}(t, \cdot)\right\|_{L^2(A_1)}+\sum_{j=1}^{2}\left\||\xi_j|\hat{v}(t, \cdot)\right\|_{L^2(A_1)} \lesssim t^{-\f{\mu}{2}}\left\|v_0\right\|_{Z, 1, 2}+t^{\delta(\varepsilon_1)-1}\left\|v_1\right\|_{Z, 1,(1+\varepsilon_1, 2)}.
\end{equation}

\vskip 0.2 true cm

{\bf Case 1.2.  The analysis  in zone $A_2$}

\vskip 0.2 true cm

In this case, it follows from  \eqref{equ:q11}, \eqref{equ:q14} and \eqref{equ:q15} that
\begin{equation}\label{equ:q36}
\begin{aligned}
& \left|\p_t\Psi_0(t, 1,\xi)\right| \lesssim|\xi|^2 t^\rho|\xi|^{\rho-1}(t|\xi|)^{-\frac{1}{2}}=t^{-\frac{\mu}{2}}|\xi|^{-\frac{\mu}{2}}|\xi|\leq (t|\xi|)^{-\f{\mu}{2}}, \\
& \left|\p_t\Psi_1(t, 1, \xi)\right|  \lesssim |\xi|t^\rho|\xi|^{\rho}(t|\xi|)^{-\frac{1}{2}}=t^{-\frac{\mu}{2}}|\xi|^{-\frac{\mu}{2}}|\xi|\leq(t|\xi|)^{-\f{\mu}{2}}.
\end{aligned}
\end{equation}
Following \eqref{equ:q25}, by the change of variables $\eta=t\xi$ and $x=ty$, we have
\begin{equation}\label{equ:q37}
\begin{aligned}
\left\|\p_t\hat{v}(t, \cdot)\right\|_{L^2(A_2)} & \leq\left\|(1+t|\xi|)^{-1}(\hat{v}_{0}+\hat{v}_{1})\right\|_{L^{2}} \\
& \lesssim t\|(v_0+v_1)(ty)\|_{H^{-1}(\R^2)}\\
& \lesssim t^{1-\f{2}{1+\varepsilon_1}}\|v_0+v_1\|_{(1+\varepsilon_1, 2)}.
\end{aligned}
\end{equation}
Due to $|\xi_j|\leq|\xi|\leq1 (j=1,2)$, then
\begin{equation}\label{equ:q38}
\left\||\xi_j|\hat{v}(t, \cdot)\right\|_{L^2(A_2)} \lesssim t^{1-\f{2}{1+\varepsilon_1}}\|v_0+v_1\|_{(1+\varepsilon_1, 2)}
\end{equation}
and
\begin{equation}\label{equ:q38}
\|\widehat{\p v}\|_{L^2(A_2)} \lesssim t^{\delta(\varepsilon_1)-1}\left\|v_0+v_1\right\|_{Z, 1,(1+\varepsilon_1, 2)}.
\end{equation}

\vskip 0.2 true cm

{\bf Case 1.3.  The analysis  in zone $A_3$}

\vskip 0.2 true cm

By $t|\xi| \leq 1$, one has from \eqref{equ:q11} and \eqref{equ:q15} that
\begin{equation}\label{equ:q39}
\begin{aligned}
& \left|\p_t\Psi_0(t, 1,\xi)\right| \lesssim|\xi|^2 t^\rho[|\xi|^{-\rho+1}(t|\xi|)^{\rho-1}+|\xi|^{\rho-1}(t|\xi|^{1-\rho})]=|\xi|^2(t^{-\mu}+t)\leq t^{-1}, \\
& \left|\p_t\Psi_1(t, 1, \xi)\right|  \lesssim |\xi|t^\rho[|\xi|^{-\rho}(t|\xi|)^{\rho-1}+|\xi|^{\rho}(t|\xi|)^{1-\rho}]=
t^{-\mu}+t|\xi|^2\leq t^{-1}.
\end{aligned}
\end{equation}
Then
\begin{equation}\label{equ:q40}
\begin{aligned}
\left\|\p_t\hat{v}\right\|_{L^2(A_3)} & \leq t^{-1}\left\|(1+t|\xi|)\times(1+t|\xi|)^{-1}(\hat{v}_{0}+\hat{v}_{1})\right\|_{L^{2}} \\
&\lesssim t^{-1}\left\|(1+t|\xi|)^{-1}(\hat{v}_{0}+\hat{v}_{1})\right\|_{L^{2}}\\
& \lesssim \|(v_0+v_1)(ty)\|_{H^{-1}(\R^2)}\\
&\lesssim t^{-\f{2}{1+\varepsilon_1}}\|v_0+v_1\|_{(1+\varepsilon_1, 2)}
\end{aligned}
\end{equation}
and
\begin{equation}\label{equ:q41}
\left\||\xi_j|\hat{v}(t, \cdot)\right\|_{L^2(A_3)} \lesssim t^{-\f{2}{1+\varepsilon_1}}\|v_0+v_1\|_{(1+\varepsilon_1, 2)}\quad\quad\text{for}~ j=1,2.
\end{equation}
Combining \eqref{equ:q40} and \eqref{equ:q41} yields
\begin{equation}\label{equ:q42}
\|\widehat{\p v}\|_{L^2(A_3)} \lesssim t^{\delta(\varepsilon_1)-1}\left\|v_0+v_1\right\|_{Z, 1,(1+\varepsilon_1, 2)}.
\end{equation}
Therefore, by Cases 1.1-1.3, we have
\begin{equation}\label{equ:q43}
\begin{aligned}
\|\p v(t, \cdot)\|_{L^2(\R^2)}&=\|\widehat{\p v}\|_{L^2(\R^2)} \leq\|\widehat{\p v}\|_{L^2(A_1)}
+\|\widehat{\p v}\|_{L^2(A_2)}+\|\widehat{\p v}\|_{L^2(A_3)}\\
&\lesssim t^{-\f{\mu}{2}}\left\|v_0\right\|_{Z, 1, 2}+t^{\delta(\varepsilon_1)-1}\left\|v_1\right\|_{Z, 1,(1+\varepsilon_1, 2)}+t^{\delta(\varepsilon_1)-1}\left\|v_0+v_1\right\|_{Z, 1,(1+\varepsilon_1, 2)}.
\end{aligned}
\end{equation}

\vskip 0.2 true cm

{\bf Case 2.  The treatment  of $\|L_0v(t, \cdot)\|_{L^2(\R^2)}$ }

\vskip 0.2 true cm

In this case, we start to show
\begin{equation}\label{equ:q44}
\|L_0v(t, \cdot)\|_{L^2(\R^2)} \lesssim t^{-\f{\mu}{2}}\left\|v_0\right\|_{Z, 1, 2}
+t^{\delta(\varepsilon_1)-1}\left\|v_1\right\|_{Z, 1,(1+\varepsilon_1, 2)}
+t^{\delta(\varepsilon_1)-1}\left\|v_0+v_1\right\|_{Z, 1,(1+\varepsilon_1, 2)}.
\end{equation}
Note that
\begin{equation}\label{equ:q45}
\begin{aligned}
\|L_0v(t, \cdot)\|_{L^2(\R^2)}& =\|t\p_t\hat{v}+\widehat{x_1\p_1v}+\widehat{x_2\p_2v}\|_{L^2(\R^2)}\\
&\leq 2\|\hat{v}\|_{L^2(\R^2)}+\|t\p_t\hat{v}\|_{L^2(\R^2)}+\|\xi_1\p_{\xi_1}\hat{v}+\xi_2\p_{\xi_2}\hat{v}\|_{L^2(\R^2)}.
\end{aligned}
\end{equation}
Thus, in order to prove \eqref{equ:q44}, it suffices to treat $\|t\p_t\hat{v}\|_{L^2(\R^2)}$ and $\|\xi_1\p_{\xi_1}\hat{v}+\xi_2\p_{\xi_2}\hat{v}\|_{L^2(\R^2)}$ since the estimate of $\|\hat{v}\|_{L^2(\R^2)}$ has been derived in \eqref{equ:q18}.

For $\|t\p_t\hat{v}\|_{L^2(\R^2)}$, it follows from \eqref{equ:q32} and \eqref{equ:q36} that for $2<\mu<3$,
\begin{equation}\label{equ:q46}
\begin{aligned}
\left\|t\p_t\hat{v}(t, \cdot)\right\|_{L^2(A_1)}
& \lesssim t^{-\frac{\mu}{2}}\left\|t|\xi|\hat{v}_{0}\right\|_{L^{2}}+t^{1-\f{\mu}{2}} \||\xi|^{-1}|\xi|\hat{v}_{1}\|_{L^{2}}\\
& \lesssim t^{-\frac{\mu}{2}}(\|L_1v_0\|_{L^2}+\|L_2v_0\|_{L^2})+ \|(1+t|\xi|)^{-1}t|\xi|\hat{v}_{1}\|_{L^{2}}\\
& \lesssim t^{-\frac{\mu}{2}}\left\|v_{0}\right\|_{Z, 1, 2}+ t^{1-\f{2}{1+\varepsilon_1}}(\|t\p_1v_1\|_{(1+\varepsilon_1, 2)}+\|t\p_2v_1\|_{(1+\varepsilon_1, 2)}) \\
& \lesssim  t^{-\frac{\mu}{2}}\left\|v_{0}\right\|_{Z, 1, 2}+ t^{1-\f{2}{1+\varepsilon_1}}\left\|v_1\right\|_{Z, 1,(1+\varepsilon_1, 2)}
\end{aligned}
\end{equation}
and
\begin{equation}\label{equ:q47}
\begin{aligned}
&\left\|t\p_t\hat{v}(t, \cdot)\right\|_{L^2(A_2)}  \leq\left\|(1+t|\xi|)^{-1}t|\xi|(\hat{v}_{0}+\hat{v}_{1})\right\|_{L^{2}}\\
& \lesssim t^{1-\f{2}{1+\varepsilon_1}}(\|t\p_1(v_0+v_1)+x_1\p_t(v_0+v_1)\|_{(1+\varepsilon_1, 2)}+\|t\p_2(v_0+v_1)+x_2\p_t(v_0+v_1)\|_{(1+\varepsilon_1, 2)})\\
& \lesssim t^{1-\f{2}{1+\varepsilon_1}}\|v_0+v_1\|_{Z,1, (1+\varepsilon_1, 2)}.
\end{aligned}
\end{equation}
Note that from \eqref{equ:q40},
\begin{equation}\label{equ:q48}
\left\|t\p_t\hat{v}(t, \cdot)\right\|_{L^2(A_3)} \lesssim t^{1-\f{2}{1+\varepsilon_1}}\|v_0+v_1\|_{(1+\varepsilon_1, 2)}.
\end{equation}
This, together with \eqref{equ:q46} and \eqref{equ:q47}, derives
\begin{equation}\label{equ:q49}
\|t\p_tv(t, \cdot)\|_{L^2(\R^2)} \lesssim t^{-\f{\mu}{2}}\left\|v_0\right\|_{Z, 1, 2}
+t^{\delta(\varepsilon_1)-1}\left\|v_1\right\|_{Z, 1,(1+\varepsilon_1, 2)}
+t^{\delta(\varepsilon_1)-1}\left\|v_0+v_1\right\|_{Z, 1,(1+\varepsilon_1, 2)}.
\end{equation}

We next treat $\|\xi_1\p_{\xi_1}\hat{v}+\xi_2\p_{\xi_2}\hat{v}\|_{L^2(\R^2)}$
and further assert
\begin{equation}\label{equ:q50}
\|\xi_1\p_{\xi_1}\hat{v}+\xi_2\p_{\xi_2}\hat{v}\|_{L^2(\R^2)} \lesssim t^{-\f{\mu}{2}}\left\|v_0\right\|_{Z, 1, 2}+t^{\delta(\varepsilon_1)-1}\left\|v_1\right\|_{Z, 1,(1+\varepsilon_1, 2)}
+t^{\delta(\varepsilon_1)-1}\left\|v_0+v_1\right\|_{Z, 1,(1+\varepsilon_1, 2)}.
\end{equation}
To prove \eqref{equ:q50}, we still
decompose the space $\R_{\xi}^2$ into three zones as in \eqref{equ:q19}.

\vskip 0.2 true cm

{\bf Case 2.1.  The analysis  in zone $A_1$}

\vskip 0.2 true cm

By \cite[Section 10]{OLBC}, one has
\begin{equation}\label{dr1}
(H_\nu^{\pm})^{\prime}(z)=H_{\nu-1}^{\pm}(z)-\frac{\nu}{z}H_\nu^{\pm},\quad
J_\nu^{\prime}(z)=J_{\nu-1}(z)-\frac{\nu}{z}J_\nu(z).
\end{equation}
It follows from \eqref{equ:q9}, \eqref{dr1} and direct computation that for $j=1,2$,
\begin{equation}\label{equ:q51}
 \begin{aligned}
& |\p_{\xi_j}\Psi_0(t, 1, \xi)|
\lesssim |\frac{\xi_j}{|\xi|} t^\rho H_{\rho-1}^{-}(|\xi|) H_\rho^{+}(t|\xi|)+\xi_j t^\rho\big(H_{\rho-2}^{-}(|\xi|)-\frac{\rho-1}{|\xi|} H_{\rho-1}^{-}(|\xi|)\big) H_\rho^{+}(t|\xi|) \\
&\quad +\xi_j t^{\rho+1}\big(H_{\rho-1}^{+}(t|\xi|)-\frac{\rho}{t|\xi|} H_\rho^{+}(t|\xi|)\big) H_{\rho-1}^{-}(|\xi|)
-\frac{\xi_j}{|\xi|} t^\rho H_{\rho-1}^{+}(|\xi|) H_\rho^{-}(t|\xi|)\\
&\quad -\xi_j t^\rho\big(H_{\rho-2}^{+}(|\xi|)-\frac{\rho-1}{|\xi|} H_{\rho-1}^{+}(|\xi|)\big) H_\rho^{-}(t|\xi|)-\xi_j t^{\rho+1}\big(H_{\rho-1}^{-}(t|\xi|)-\frac{\rho}{t|\xi|} H_\rho^{-}(t|\xi|)\big) H_{\rho-1}^{+}(|\xi|)|\\
& \lesssim \frac{|\xi_j|}{|\xi|} t^\rho H_{\rho-1}^{-}(|\xi|) H_\rho^{+}(t|\xi|)+|\xi_j| t^\rho H_{\rho-2}^{-}(|\xi|) H_\rho^{+}(t|\xi|)+|\xi_j| t^{\rho+1} H_{\rho-1}^{+}(t|\xi|) H_{\rho-1}^{-}(|\xi|) \\
&\quad +\frac{|\xi_j|}{|\xi|} t^\rho H_{\rho-1}^{+}(|\xi|) H_\rho^{-}(t|\xi|)+|\xi_j| t^\rho H_{\rho-2}^{+}(|\xi|) H_\rho^{-}(t|\xi|)+|\xi_j| t^{\rho+1} H_{\rho-1}^{-}(t|\xi|) H_{\rho-1}^{+}(|\xi|)
\end{aligned}
\end{equation}
and
\begin{equation}\label{equ:q52}
\begin{aligned}
&|\p_{\xi_j}\Psi_1(t, 1, \xi)|
\lesssim \frac{|\xi_j|}{|\xi|^2} t^\rho H_{\rho}^{-}(|\xi|) H_\rho^{+}(t|\xi|)+\f{|\xi_j|}{|\xi|} t^\rho H_{\rho-1}^{-}(|\xi|) H_\rho^{+}(t|\xi|)+\f{|\xi_j|}{|\xi|} t^{\rho+1} H_{\rho-1}^{+}(t|\xi|) H_{\rho}^{-}(|\xi|) \\
&\quad +\frac{|\xi_j|}{|\xi|^2} t^\rho H_{\rho}^{+}(|\xi|) H_\rho^{-}(t|\xi|)+\f{|\xi_j|}{|\xi|} t^\rho H_{\rho-1}^{+}(|\xi|) H_\rho^{-}(t|\xi|)+\f{|\xi_j|}{|\xi|} t^{\rho+1} H_{\rho-1}^{-}(t|\xi|) H_{\rho}^{+}(|\xi|).
\end{aligned}
\end{equation}
Therefore, as in \eqref{equ:q33}, it follows from \eqref{equ:q51}-\eqref{equ:q52}, \eqref{equ:q4} and  \eqref{equ:q12} that for $2<\mu<3$,
\begin{equation*}\label{equ:q61}
\begin{aligned}
&\big\|\xi_1\p_{\xi_1}\hat{v}+\xi_2\p_{\xi_2}\hat{v}\big\|_{L^2(A_1)}
\lesssim \big\|t^{-\frac{\mu}{2}}\hat{v}_{0}(\xi)+t^{-\frac{\mu}{2}}t(|\xi_1|+|\xi_2|)\hat{v}_{0}(\xi)\big\|_{L^{2}}+ \big\|t^{-\f{\mu}{2}}(\xi_1\p_{\xi_1}+\xi_2\p_{\xi_2})\hat{v}_{0}(\xi)\big\|_{L^{2}}\\
&\quad +\big\|t^{-\frac{\mu}{2}}(1+|\xi|^2)^{-\f{1}{2}}t(|\xi_1|+|\xi_2|)\hat{v}_{1}(\xi)\big\|_{L^{2}}
+\big\|t^{-\f{\mu}{2}}(1+|\xi|^2)^{-\f{1}{2}}(\xi_1\p_{\xi_1}+\xi_2\p_{\xi_2})\hat{v}_{1}(\xi)\big\|_{L^{2}}\\
\end{aligned}
\end{equation*}

\begin{equation}\label{equ:q61}
\begin{aligned}
& \lesssim t^{-\frac{\mu}{2}}\left\|v_{0}\right\|_{L^{2}}+ t^{-\f{\mu}{2}} \left\|t\p_1v_0\right\|_{L^{2}}+t^{-\f{\mu}{2}} \left\|t\p_2v_0\right\|_{L^{2}}+ t^{-\f{\mu}{2}}\left\|(x_1\p_1+x_2\p_2)v_0\right\|_{L^{2}}+ t^{-\f{\mu}{2}+2-\f{2}{1+\varepsilon_1}}\left\|t\p_1v_1\right\|_{(1+\varepsilon_1, 2)}\\
&\quad +t^{-\f{\mu}{2}+2-\f{2}{1+\varepsilon_1}}\left\|t\p_2v_1\right\|_{(1+\varepsilon_1, 2)} + t^{-\f{\mu}{2}+2-\f{2}{1+\varepsilon_1}}\left\|(x_1\p_1+x_2\p_2)v_1\right\|_{(1+\varepsilon_1, 2)}+t^{-\f{\mu}{2}+2-\f{2}{1+\varepsilon_1}}\left\|v_1\right\|_{(1+\varepsilon_1, 2)}\\
&\lesssim  t^{-\frac{\mu}{2}}\left\|v_{0}\right\|_{Z, 1, 2}+ t^{-\f{\mu}{2}+2-\f{2}{1+\varepsilon_1}}\left\|v_1\right\|_{Z, 1,(1+\varepsilon_1, 2)}\lesssim t^{-\frac{\mu}{2}}\left\|v_{0}\right\|_{Z, 1, 2}+ t^{1-\f{2}{1+\varepsilon_1}}\left\|v_1\right\|_{Z, 1,(1+\varepsilon_1, 2)}.
\end{aligned}
\end{equation}
\vskip 0.2 true cm

{\bf Case 2.2.  The analysis  in zone $A_2$}

\vskip 0.2 true cm

By \eqref{equ:q11} and \eqref{dr1}, as treated in \eqref{equ:q51}-\eqref{equ:q52}, we obtain that for $j=1,2$,
\begin{equation}\label{equ:q55}
\begin{aligned}
& |\p_{\xi_j}\Psi_0(t, 1, \xi)|
\lesssim \frac{|\xi_j|}{|\xi|} t^\rho J_{-\rho+1}(|\xi|) J_\rho(t|\xi|)+\frac{|\xi_j|}{|\xi|} t^\rho J_{\rho-1}(|\xi|) J_{-\rho}(t|\xi|)
+|\xi_j| t^\rho J_{-\rho}(|\xi|) J_\rho(t|\xi|)\\
&\quad +t^{\rho}\f{|\xi_j|}{|\xi|} J_{-\rho+1}(|\xi|) J_{\rho}(t|\xi|) +t^{\rho+1}|\xi_j| J_{\rho-1}(t|\xi|) J_{-\rho+1}(|\xi|)
+t^\rho|\xi_j| J_{\rho-2}(|\xi|) J_{-\rho}(t|\xi|)\\
&\quad +t^{\rho+1}|\xi_j| J_{-\rho-1}(t|\xi|) J_{\rho-1}(|\xi|)+t^{\rho}\f{|\xi_j|}{|\xi|} J_{-\rho}(t|\xi|) J_{\rho-1}(|\xi|)
\end{aligned}
\end{equation}
and
\begin{equation}\label{equ:q56}
\begin{aligned}
&|\p_{\xi_j}\Psi_1(t, 1, \xi)|
\lesssim \frac{|\xi_j|}{|\xi|} t^\rho J_{-\rho-1}(|\xi|) J_\rho(t|\xi|)+\frac{|\xi_j|}{|\xi|} t^{\rho+1} J_{\rho-1}(t|\xi|) J_{-\rho}(|\xi|)+\f{|\xi_j|}{|\xi|} t^\rho J_{\rho-1}(|\xi|) J_{-\rho}(t|\xi|)\\
&\quad +t^{\rho+1}\f{|\xi_j|}{|\xi|} J_{-\rho-1}(t|\xi|) J_{\rho}(|\xi|) +t^{\rho}\f{|\xi_j|}{|\xi|^2} J_{\rho}(t|\xi|) J_{-\rho}(|\xi|)+t^\rho\f{|\xi_j|}{|\xi|^2} J_{-\rho}(|\xi|) J_{\rho}(t|\xi|)\\
&\quad +t^{\rho}\f{|\xi_j|}{|\xi|^2} J_{-\rho}(t|\xi|) J_{\rho}(|\xi|)+t^{\rho}\f{|\xi_j|}{|\xi|^2} J_{-\rho}(t|\xi|) J_{\rho}(|\xi|).
\end{aligned}
\end{equation}
It follows from \eqref{equ:q55}-\eqref{equ:q56} and \eqref{equ:q14}-\eqref{equ:q15} that for $2<\mu<3$,
\begin{equation}\label{equ:q59}
\begin{aligned}
&\left\|\xi_1\p_{\xi_1}\hat{v}+\xi_2\p_{\xi_2}\hat{v}\right\|_{L^2(A_2)}
\lesssim \big\|t^{-\frac{\mu}{2}}\hat{v}_{0}(\xi)\big\|_{L^{2}}
+\big\|(1+t|\xi|)^{-\frac{\mu}{2}}t(|\xi_1|+|\xi_2|)(\hat{v}_{0}(\xi)+\hat{v}_{1}(\xi))\big\|_{L^{2}}\\
&\qquad +\big\|(1+t|\xi|)^{-\f{\mu}{2}}(\xi_1\p_{\xi_1}+\xi_2\p_{\xi_2})(\hat{v}_{0}(\xi)+\hat{v}_{1}(\xi))\big\|_{L^{2}}\\
&\lesssim  t^{-\frac{\mu}{2}}\left\|v_{0}\right\|_{L^{2}}+t^{1-\f{2}{1+\varepsilon_1}}\left\|t\p_1(v_0+v_1)\right\|_{(1+\varepsilon_1, 2)}+t^{1-\f{2}{1+\varepsilon_1}}\left\|t\p_2(v_0+v_1)\right\|_{(1+\varepsilon_1, 2)}\\
&\qquad+ t^{1-\f{2}{1+\varepsilon_1}}\left\|(x_1\p_1+x_2\p_2)(v_0+v_1)\right\|_{(1+\varepsilon_1, 2)}+t^{1-\f{2}{1+\varepsilon_1}}\left\|v_0+v_1\right\|_{(1+\varepsilon_1, 2)}\\
&\lesssim t^{-\frac{\mu}{2}}\left\|v_{0}\right\|_{L^{2}}+ t^{1-\f{2}{1+\varepsilon_1}}\left\|v_0+v_1\right\|_{Z, 1, (1+\varepsilon_1, 2)}.
\end{aligned}
\end{equation}

\vskip 0.2 true cm

{\bf Case 2.3.  The analysis  in zone $A_3$}

\vskip 0.2 true cm

As in the zone $A_2$, by \eqref{equ:q55}-\eqref{equ:q56} and \eqref{equ:q15}, one can obtain that for $t|\xi|\leq 1$,
\begin{equation*}\label{equ:q60}
\begin{aligned}
\left\|\xi_1\p_{\xi_1}\hat{v}+\xi_2\p_{\xi_2}\hat{v}\right\|_{L^2(A_3)}
&\lesssim \big\|t^{-\frac{\mu}{2}}\hat{v}_{0}(\xi)\big\|_{L^{2}}+\big\|(1+t|\xi|)(1+t|\xi|)^{-1}(\hat{v}_{0}(\xi)+\hat{v}_{1}(\xi))\big\|_{L^{2}}\\
&+\big\|(1+t|\xi|)(1+t|\xi|)^{-1}(\xi_1\p_{\xi_1}+\xi_2\p_{\xi_2})(\hat{v}_{0}(\xi)+\hat{v}_{1}(\xi))\big\|_{L^{2}}\\
\end{aligned}
\end{equation*}
\begin{equation}\label{equ:q60}
\begin{aligned}
&\lesssim  t^{-\frac{\mu}{2}}\big\|v_{0}\big\|_{L^{2}}+ t^{1-\f{2}{1+\varepsilon_1}}\big\|v_0+v_1\big\|_{(1+\varepsilon_1, 2)}\\
&\quad+t^{1-\f{2}{1+\varepsilon_1}}\big\|(x_1\p_1+x_2\p_2)(v_0+v_1)\big\|_{(1+\varepsilon_1, 2)}\\
&\lesssim t^{-\frac{\mu}{2}}\big\|v_{0}\big\|_{L^{2}}+ t^{1-\f{2}{1+\varepsilon_1}}\big\|v_0+v_1\big\|_{Z, 1, (1+\varepsilon_1, 2)}.
\end{aligned}
\end{equation}

Therefore, combining \eqref{equ:q61}, \eqref{equ:q59} and \eqref{equ:q60},
the assertion \eqref{equ:q50} and further \eqref{equ:q44} are proved.

\vskip 0.2 true cm

{\bf Case 3.  The treatment  of $\|L_jv(t, \cdot)\|_{L^2(\R^2)}(j=1, 2)$ }

\vskip 0.2 true cm

We assert
\begin{equation}\label{equ:q62}
\|L_jv(t, \cdot)\|_{L^2(\R^2)} \lesssim t^{-\f{\mu}{2}}\left\|v_0\right\|_{Z, 1, 2}+t^{\delta(\varepsilon_1)-1}\left\|v_1\right\|_{Z, 1,(1+\varepsilon_1, 2)}+t^{\delta(\varepsilon_1)-1}\left\|v_0+v_1\right\|_{Z, 1,(1+\varepsilon_1, 2)}.
\end{equation}

To prove \eqref{equ:q62}, we only need to deal with $\|L_1v(t, \cdot)\|_{L^2(\R^2)}$ since the treatment on $\|L_2v(t, \cdot)\|_{L^2(\R^2)}$ is completely analogous. It follows from $t\geq1$ and the support condition of the initial data that $|x|\lesssim t$ holds
for $(t, x)\in \operatorname{supp}v$. Hence,
\begin{equation}\label{equ:q63}
\begin{aligned}
\|L_1v(t, \cdot)\|_{L^2(\R^2)}\leq \|t\p_1v\|_{L^2(\R^2)}+\|x_1\p_tv\|_{L^2(\R^2)}
\lesssim \|t|\xi_1|\hat{v}\|_{L^2(\R^2)}+\|t\p_t\hat{v}\|_{L^2(\R^2)}.
\end{aligned}
\end{equation}
Note that $\|t\p_t\hat{v}\|_{L^2(\R^2)}$ can be estimated as in \eqref{equ:q49},
it suffices only to treat $\|t|\xi_1|\hat{v}\|_{L^2(\R^2)}$.

It follows from  \eqref{equ:q9} and  \eqref{equ:q12} that for $|\xi|\geq1$,
$$
\begin{aligned}
& \left|\Psi_0(t, 1,\xi)\right| \lesssim t^{-\frac{\mu}{2}}, \\
& \left|\Psi_1(t, 1, \xi)\right|  \lesssim t^{-\frac{\mu}{2}}|\xi|^{-1}.
\end{aligned}
$$
Then
\begin{equation}\label{equ:q65}
\begin{aligned}
\left\|t|\xi_1|\hat{v}(t, \cdot)\right\|_{L^2(A_1)} & \lesssim t^{-\frac{\mu}{2}}\left\|t|\xi_1|\hat{v}_{0}\right\|_{L^{2}}+t^{-\f{\mu}{2}} \|(1+|\xi|^2)^{-\f{1}{2}}t|\xi_1|\hat{v}_{1}\|_{L^{2}}\\
& \lesssim t^{-\frac{\mu}{2}}\left\|t\p_1v_{0}\right\|_{L^{2}}+ t^{-\f{\mu}{2}+2-\f{2}{1+\varepsilon_1}}\left\|t\p_1v_{1}\right\|_{(1+\varepsilon_1, 2)}\\
& \lesssim t^{-\frac{\mu}{2}}\left\|v_{0}\right\|_{Z, 1, 2}+ t^{1-\f{2}{1+\varepsilon_1}}\left\|v_{1}\right\|_{Z, 1, (1+\varepsilon_1, 2)}.
\end{aligned}
\end{equation}
On the other hand, for $|\xi|\leq1$ and $t|\xi|\geq1$, it follows from \eqref{equ:q11}, \eqref{equ:q14} and \eqref{equ:q15} that for $2<\mu<3$,
$$
\begin{aligned}
& \left|\Psi_0(t, 1,\xi)\right| \lesssim t^{-\frac{\mu}{2}}|\xi|^{-\frac{\mu}{2}}\lesssim (1+t|\xi|)^{-1} ,\\
& \left|\Psi_1(t, 1, \xi)\right|  \lesssim t^{-\frac{\mu}{2}}|\xi|^{-\frac{\mu}{2}}\lesssim (1+t|\xi|)^{-1}.
\end{aligned}
$$
Then
\begin{equation}\label{equ:q67}
\begin{aligned}
\left\|t|\xi_1|\hat{v}(t, \cdot)\right\|_{L^2(A_2)}&\lesssim  \|(1+t|\xi|)^{-1}t|\xi_1|(\hat{v}_{0}+\hat{v}_{1})\|_{L^{2}}\\
& \lesssim t^{1-\f{2}{1+\varepsilon_1}}\left\|t\p_1(v_0+v_{1})\right\|_{(1+\varepsilon_1, 2)}\\
& \lesssim  t^{1-\f{2}{1+\varepsilon_1}}\left\|(v_0+v_{1})\right\|_{Z, 1, (1+\varepsilon_1, 2)}.
\end{aligned}
\end{equation}
For $t|\xi|\leq1$, by \eqref{equ:q11} and \eqref{equ:q14}, we have
$$
\begin{aligned}
& \left|\Psi_0(t, 1,\xi)\right| \lesssim 1= (1+t|\xi|)(1+t|\xi|)^{-1}\lesssim  (1+t|\xi|)^{-1},\\
& \left|\Psi_1(t, 1, \xi)\right|  \lesssim 1= (1+t|\xi|)(1+t|\xi|)^{-1}\lesssim  (1+t|\xi|)^{-1}.
\end{aligned}
$$
Then
\begin{equation}\label{equ:q68}
\begin{aligned}
\left\|t|\xi_1|\hat{v}(t, \cdot)\right\|_{L^2(A_3)}\lesssim  \|(1+t|\xi|)^{-1}(\hat{v}_{0}+\hat{v}_{1})\|_{L^{2}}
\lesssim t^{1-\f{2}{1+\varepsilon_1}}\left\|v_0+v_{1}\right\|_{(1+\varepsilon_1, 2)}.
\end{aligned}
\end{equation}

Therefore, combining \eqref{equ:q65}, \eqref{equ:q67} and \eqref{equ:q68} gives
\begin{equation}\label{equ:q69}
\|t|\xi_1|\hat{v}(t, \cdot)\|_{L^2(\R^2)} \lesssim t^{-\f{\mu}{2}}\left\|v_0\right\|_{Z, 1, 2}
+t^{\delta(\varepsilon_1)-1}\left\|v_1\right\|_{Z, 1,(1+\varepsilon_1, 2)}
+t^{\delta(\varepsilon_1)-1}\left\|v_0+v_1\right\|_{Z, 1,(1+\varepsilon_1, 2)}.
\end{equation}
Together with \eqref{equ:q49}, this yields \eqref{equ:q62} immediately.
\vskip 0.2 true cm

{\bf Case 4.  The treatment of $\|\Omega_{12}v(t, \cdot)\|_{L^2(\R^2)}$ }

\vskip 0.2 true cm
We now start to prove
\begin{equation}\label{equ:q70}
\|\Omega_{12}v(t, \cdot)\|_{L^2(\R^2)} \lesssim t^{-\f{\mu}{2}}\left\|v_0\right\|_{Z, 1, 2}+t^{\delta(\varepsilon_1)-1}\left\|v_1\right\|_{Z, 1,(1+\varepsilon_1, 2)}+t^{\delta(\varepsilon_1)-1}\left\|v_0+v_1\right\|_{Z, 1,(1+\varepsilon_1, 2)}.
\end{equation}

As treated in \eqref{equ:q63}, one has
\begin{equation}\label{equ:q71}
\begin{aligned}
\|\Omega_{12}v(t, \cdot)\|_{L^2(\R^2)}&\leq \|\f{x_1}{t}t\p_2v\|_{L^2(\R^2)}+\|\f{x_2}{t}t\p_1v\|_{L^2(\R^2)}\\
& \lesssim \|t|\xi_2|\hat{v}\|_{L^2(\R^2)}+\|t|\xi_1|\hat{v}\|_{L^2(\R^2)}.\\
\end{aligned}
\end{equation}
Note that by \eqref{equ:q69}, it holds that for $j=1,2$,
$$
\|t|\xi_j|\hat{v}(t, \cdot)\|_{L^2(\R^2)} \lesssim t^{-\f{\mu}{2}}\left\|v_0\right\|_{Z, 1, 2}
+t^{\delta(\varepsilon_1)-1}\left\|v_1\right\|_{Z, 1,(1+\varepsilon_1, 2)}
+t^{\delta(\varepsilon_1)-1}\left\|v_0+v_1\right\|_{Z, 1,(1+\varepsilon_1, 2)}.
$$
Therefore, \eqref{equ:q70} is shown.

Finally, collecting the results in Cases 1-4, Lemma \ref{lem2} is proved.
\end{proof}

Based on some crucial estimates in Lemma \ref{lem2},  we now treat the first order derivative of solution $v$
to \eqref{equ:q16}.
\begin{lemma}\label{lem3}
Let $v$ be the solution to \eqref{equ:q16}. Then for $2<\mu<3$,
\begin{equation}\label{equ:q74}
\|\p v(t, \cdot)\|_{Z, 1, 2} \lesssim t^{-\f{\mu}{2}}\left\|\nabla v_0\right\|_{Z, 1, 2}+t^{-\f{\mu}{2}}\left\| v_1\right\|_{Z, 1, 2}+t^{-1}\left\|v_0+v_1\right\|_{Z, 1,2}.
\end{equation}
\end{lemma}
\begin{proof}
By the definition of $\|\p v\|_{Z, 1, 2}$,  we firstly prove
\begin{equation}\label{equ:q75}
\|\p v(t, \cdot)\|_{L^2} \lesssim t^{-\f{\mu}{2}}\left\|\nabla v_0\right\|_{L^ 2}+t^{-\f{\mu}{2}}\left\| v_1\right\|_{L^ 2}+t^{-1}\left\|v_0+v_1\right\|_{L^ 2}.
\end{equation}

In fact, as in Lemma \ref{lem2}, we decompose the frequency space $\Bbb R_{\xi}^2$ into three zones
$A_1$, $A_2$ and $A_3$. For $|\xi|\geq 1$, it follows from \eqref{equ:q9}, \eqref{equ:q20} and Parseval's equality that
\begin{equation}\label{equ:q76}
\begin{aligned}
\left\||\xi|\hat{v}(t, \cdot)\right\|_{L^2(A_1)} & \leq\left\|  \Psi_{0}(t, 1, \xi)\right\|_{L^{\infty}}\left\||\xi|\hat{v}_{0}\right\|_{L^{2}}+\left\||\Psi_{1 }(t, 1, \xi)||\xi|\hat{v}_{1}\right\|_{L^{2}} \\
& \lesssim t^{-\frac{\mu}{2}}\left\||\xi| \hat{v}_{0}\right\|_{L^{2}}+t^{-\f{\mu}{2}} \|\hat{v}_{1}\|_{L^{2}}\\
& = t^{-\frac{\mu}{2}}\left\|\nabla v_{0}\right\|_{L^{2}}+t^{-\f{\mu}{2}} \|v_{1}\|_{L^{2}}.
\end{aligned}
\end{equation}
For $|\xi|\leq1$ and $t|\xi|\geq1$, by \eqref{equ:q11} and \eqref{equ:q28}, one has
\begin{equation}\label{equ:q77}
\left\||\xi|\hat{v}(t, \cdot)\right\|_{L^2(A_2)}  \lesssim \big\|(t|\xi|)^{-\f{\mu}{2}}|\xi|(\hat{v}_{0}+\hat{v}_{1})\big\|_{L^{2}}
\lesssim t^{-1} \big\|(t|\xi|)^{1-\f{\mu}{2}}(\hat{v}_{0}+\hat{v}_{1})\big\|_{L^{2}} \lesssim t^{-1}\left\|v_0+v_1\right\|_{L^ 2}.
\end{equation}
For $t|\xi|\leq1$, by \eqref{equ:q11} and \eqref{equ:q26}, we arrive at
$$
\left\||\xi|\hat{v}(t, \cdot)\right\|_{L^2(A_3)} \lesssim \left\||\xi|(\hat{v}_{0}+\hat{v}_{1})\right\|_{L^{2}}  \leq t^{-1}\left\|v_{0}+v_{1}\right\|_{L^{2}}.
$$
This, together with \eqref{equ:q76} and \eqref{equ:q77}, yields
$$
\left\||\xi|\hat{v}(t, \cdot)\right\|_{L^2(\R^2)}\lesssim t^{-\f{\mu}{2}}\left\|\nabla v_0\right\|_{L^ 2}+t^{-\f{\mu}{2}}\left\| v_1\right\|_{L^2}+t^{-1}\left\|v_0+v_1\right\|_{L^ 2}.
$$
Analogously, by \eqref{equ:q32}, \eqref{equ:q36} and \eqref{equ:q39}, we have
$$
\left\|\p_t\hat{v}(t, \cdot)\right\|_{L^2(\R^2)}\lesssim t^{-\f{\mu}{2}}\left\|\nabla v_0\right\|_{L^ 2}+t^{-\f{\mu}{2}}\left\| v_1\right\|_{L^ 2}+t^{-1}\left\|v_0+v_1\right\|_{L^ 2}.
$$
Thus, \eqref{equ:q75} is shown.

Next we turn to the proof of \eqref{equ:q74}. As treated as in  Lemma \ref{lem2}, the proof
procedure is divided into the following four cases.

\vskip 0.2 true cm

{\bf Case 1.  The treatment  of $\|\p(\p v)(t, \cdot)\|_{L^2(\R^2)}$ }

\vskip 0.2 true cm

We now show
\begin{equation}\label{equ:q78}
\|\p(\p v)(t, \cdot)\|_{L^2(\R^2)} \lesssim t^{-\f{\mu}{2}}\left\|\nabla v_0\right\|_{Z, 1, 2}
+t^{-\f{\mu}{2}}\left\|v_1\right\|_{Z, 1,2}+t^{-1}\left\|v_0+v_1\right\|_{Z, 1, 2}.
\end{equation}
Note that
\begin{equation}\label{equ:q78-1}
\begin{aligned}
\|\p(\p v)(t, \cdot)\|_{L^2(\R^2)} & \lesssim\|\p_t^2v\|_{L^2(\R^2)}+\sum_{j=1}^2\|\p_t\p_jv\|_{L^2(\R^2)}+\sum_{i,j=1}^2\|\p_i\p_jv\|_{L^2(\R^2)}\\
& \lesssim \|\p_t^2\hat{v}\|_{L^2(\R^2)}+\left\||\xi|\p_t\hat{v}(t, \cdot)\right\|_{L^2(\R^2)}+ \sum_{i,j=1}^2\left\||\xi_j||\xi_i|\hat{v}(t, \cdot)\right\|_{L^2(\R^2)}.
\end{aligned}
\end{equation}
It follows from Lemma \ref{lem6} and Lemma \ref{lem1} that for $t|\xi|\leq1$,
\begin{equation}\label{p-1}
\begin{aligned}
& \left|\p_t^2\Psi_0(t, 1,\xi)\right| \lesssim t^\rho|\xi|^3(t|\xi|)^{-\f{1}{2}}|\xi|^{-\f{1}{2}}= t^{-\frac{\mu}{2}}|\xi|\cdot|\xi|, \\
& \left|\p_t^2\Psi_1(t, 1, \xi)\right|  \lesssim t^\rho|\xi|^2(t|\xi|)^{-\f{1}{2}}|\xi|^{-\f{1}{2}}= t^{-\frac{\mu}{2}}\cdot|\xi|.
\end{aligned}
\end{equation}
For  $|\xi|\leq1$ and $t|\xi|\geq1$,
\begin{equation}\label{p-2}
\begin{aligned}
& \left|\p_t^2\Psi_0(t, 1,\xi)\right| \lesssim t^\rho|\xi|^3(t|\xi|)^{-\f{1}{2}}|\xi|^{\rho-1}\leq (t|\xi|)^{-\frac{\mu}{2}}\cdot|\xi|, \\
& \left|\p_t^2\Psi_1(t, 1, \xi)\right|  \lesssim t^\rho|\xi|^2(t|\xi|)^{-\f{1}{2}}|\xi|^{\rho}\leq (t|\xi|)^{-\frac{\mu}{2}}\cdot|\xi|.
\end{aligned}
\end{equation}
For $t|\xi|\leq1$,
\begin{equation}\label{p-3}
\begin{aligned}
& \left|\p_t^2\Psi_0(t, 1,\xi)\right| \lesssim t^{-1}|\xi|, \\
& \left|\p_t^2\Psi_1(t, 1, \xi)\right|  \lesssim t^{-1}|\xi|.
\end{aligned}
\end{equation}
By comparing the estimates in \eqref{p-1}, \eqref{p-2} and \eqref{p-3} with the corresponding ones in \eqref{equ:q32}, \eqref{equ:q36} and \eqref{equ:q39},
it is known that the treatment of $\|\p_t^2\hat{v}\|_{L^2(\R^2)}$ can be similar to the one for $\left\||\xi|\p_t\hat{v}(t, \cdot)\right\|_{L^2(\R^2)}$. Therefore, in order to prove  \eqref{equ:q78}, it suffices only to deal with $\left\||\xi|\p_t\hat{v}(t, \cdot)\right\|_{L^2(\R^2)}$ and $\left\||\xi_j||\xi_i|\hat{v}(t, \cdot)\right\|_{L^2(\R^2)}$ $(i, j=1, 2)$.

It follows from  \eqref{equ:q32} and \eqref{equ:q20} that for $|\xi|\geq 1$,
$$
\begin{aligned}
\left\||\xi|\p_t\hat{v}(t, \cdot)\right\|_{L^2(A_1)}
& \lesssim t^{-\frac{\mu}{2}}\left\||\xi||\xi|\hat{v_{0}}\right\|_{L^{2}}+t^{-\f{\mu}{2}} \||\xi|\hat{v}_{1}\|_{L^{2}}\\
& \lesssim t^{-\frac{\mu}{2}}\left\|\nabla v_{0}\right\|_{Z, 1, 2}+t^{-\f{\mu}{2}}\left\|(|\xi_1|+|\xi_2|)\hat{v}_1\right\|_{L^2}\\
& \lesssim t^{-\f{\mu}{2}}\left\|\nabla v_0\right\|_{Z, 1, 2}+t^{-\f{\mu}{2}}\left\|v_1\right\|_{Z, 1,2}
\end{aligned}
$$
and
$$
\begin{aligned}
\left\||\xi_j||\xi_i|\hat{v}(t, \cdot)\right\|_{L^2(A_1)}
& \lesssim t^{-\frac{\mu}{2}}\left\||\xi||\xi_j|v_{0}\right\|_{L^{2}}+t^{-\f{\mu}{2}} \||\xi_j|\hat{v}_{1}\|_{L^{2}}\\
& \lesssim t^{-\f{\mu}{2}}\left\|\nabla v_0\right\|_{Z, 1, 2}+t^{-\f{\mu}{2}}\left\|v_1\right\|_{Z,1,2}.
\end{aligned}
$$
For $|\xi|\leq1$ and $t|\xi|\geq1$, by \eqref{equ:q28} and \eqref{equ:q36}, one has that for $j=1, 2$,
$$
\left\||\xi_j||\xi_i|\hat{v}(t, \cdot)\right\|_{L^2(A_2)}\lesssim \big\|(t|\xi|)^{-\f{\mu}{2}}|\xi|(\hat{v}_{0}+\hat{v}_{1})\big\|_{L^{2}}
\lesssim t^{-1} \big\|(t|\xi|)^{1-\f{\mu}{2}}(\hat{v}_{0}+\hat{v}_{1})\big\|_{L^{2}} \lesssim t^{-1}\left\|v_0+v_1\right\|_{L^ 2}
$$
and
$$
\left\||\xi|\p_t\hat{v}(t, \cdot)\right\|_{L^2(A_2)}\lesssim \big\|(t|\xi|)^{-\f{\mu}{2}}|\xi|(\hat{v}_{0}+\hat{v}_{1})\big\|_{L^{2}} \lesssim t^{-1}\left\|v_0+v_1\right\|_{L^ 2}.
$$
For $t|\xi|\leq1$, by \eqref{equ:q26} and \eqref{equ:q39}, we get
$$
\left\||\xi_j||\xi_i|\hat{v}(t, \cdot)\right\|_{L^2(A_3)}\lesssim \left\||\xi|^2(\hat{v}_{0}+\hat{v}_{1})\right\|_{L^{2}} \lesssim t^{-2}\left\|v_0+v_1\right\|_{L^ 2}
$$
and
$$
\left\||\xi|\p_t\hat{v}(t, \cdot)\right\|_{L^2(A_3)}\lesssim \left\|t^{-1}|\xi|(\hat{v}_{0}+\hat{v}_{1})\right\|_{L^{2}} \lesssim t^{-2}\left\|v_0+v_1\right\|_{L^ 2}.
$$
Collecting the above conclusions, \eqref{equ:q78} is then proved.
\vskip 0.2 true cm

{\bf Case 2.  The treatment of $\|L_0(\p v)(t, \cdot)\|_{L^2(\R^2)}$}

\vskip 0.2 true cm

We now prove
\begin{equation}\label{equ:q79}
\|L_0(\p v)(t, \cdot)\|_{L^2(\R^2)} \lesssim t^{-\f{\mu}{2}}\left\|\nabla v_0\right\|_{Z, 1, 2}
+t^{-\f{\mu}{2}}\left\|v_1\right\|_{Z, 1,2}+t^{-1}\left\|v_0+v_1\right\|_{Z, 1, 2}.
\end{equation}
Analogous to the treatment of \eqref{equ:q78-1} and \eqref{equ:q45}, it suffices to estimate  $\|t|\xi|\p_t\hat{v}\|_{L^2(\R^2)}$ and $\||\xi|(\xi_1\p_{\xi_1}\hat{v}+\xi_2\p_{\xi_2})\hat{v}\|_{L^2(\R^2)}$.

At first, we estimate $\|t|\xi|\p_t\hat{v}\|_{L^2(\R^2)}$. For $|\xi|\geq1$,  by \eqref{equ:q32}, one has
\begin{equation}\label{equ:q80}
\left\|t|\xi|\p_t\hat{v}(t, \cdot)\right\|_{L^2(A_1)}
 \lesssim t^{-\frac{\mu}{2}}\left\|t|\xi||\xi|\hat{v}_{0}\right\|_{L^{2}}+t^{-\f{\mu}{2}} \|t|\xi|\hat{v}_{1}\|_{L^{2}} \lesssim t^{-\f{\mu}{2}}\left\|\nabla v_0\right\|_{Z,1,2}+t^{-\f{\mu}{2}}\left\|v_1\right\|_{Z,1,2}.
\end{equation}
For $|\xi|\leq1$ and $t|\xi|\geq1$,  by  \eqref{equ:q36}, we arrive at
\begin{equation}\label{equ:q81}
\begin{aligned}
\left\|t|\xi|\p_t\hat{v}(t, \cdot)\right\|_{L^2(A_2)}
&\lesssim t^{-1}\big\|(t|\xi|)^{1-\f{\mu}{2}}t|\xi|(\hat{v}_{0}+\hat{v}_{1})\big\|_{L^2}\\
& \lesssim  t^{-1}\left\|(t\p_1+x_1\p_t)(v_{0}+v_1)\right\|_{L^{2}}+ t^{-1}\left\|(t\p_1+x_1\p_t)(v_0+v_{1})\right\|_{L^2}\\
&\lesssim t^{-1}\left\|v_{0}+v_{1}\right\|_{Z, 1, 2}.
\end{aligned}
\end{equation}
For  $t|\xi|\leq1$,  by  \eqref{equ:q39}, one can obtain
$$
\left\|t|\xi|\p_t\hat{v}(t, \cdot)\right\|_{L^2(A_3)}  \leq\big\|t^{-1}t|\xi|(\hat{v}_{0}+\hat{v}_{1})\big\|_{L^{2}}
\lesssim t^{-1}\left\|v_{0}+v_{1}\right\|_{L^ 2}.
$$
This, together with \eqref{equ:q80} and \eqref{equ:q81}, yields
\begin{equation}\label{equ:q82}
\|t|\xi|\p_t\hat{v}\|_{L^2(\R^2)}\lesssim t^{-\f{\mu}{2}}\left\|\nabla v_0\right\|_{Z, 1, 2}
+t^{-\f{\mu}{2}}\left\|v_1\right\|_{Z, 1,2}+t^{-1}\left\|v_0+v_1\right\|_{Z, 1, 2}.
\end{equation}

Next we estimate  $\||\xi|(\xi_1\p_{\xi_1}\hat{v}+\xi_2\p_{\xi_2})\hat{v}\|_{L^2(\R^2)}$.

For $|\xi|\geq1$, it follows from \eqref{dr1}-\eqref{equ:q51} and \eqref{equ:q12} that
\begin{equation}\label{equ:q83}
\begin{aligned}
&\left\||\xi|(\xi_1\p_{\xi_1}\hat{v}+\xi_2\p_{\xi_2})\hat{v}\right\|_{L^2(A_1)}\\
&=\||\xi|\xi_1(\p_{\xi_1}\Psi_0)\hat{v}_0(\xi)+|\xi|\xi_1\Psi_0\p_{\xi_1}\hat{v}_0(\xi)+|\xi|\xi_1(\p_{\xi_1}\Psi_1)\hat{v}_1(\xi)
+|\xi|\xi_1\Psi_1\p_{\xi_1}\hat{v}_1(\xi)\\
&\quad +|\xi|\xi_2(\p_{\xi_2}\Psi_0)\hat{v}_0(\xi)+|\xi|\xi_2\Psi_0\p_{\xi_2}\hat{v}_0(\xi)
+|\xi|\xi_2(\p_{\xi_2}\Psi_1)\hat{v}_1(\xi)+|\xi|\xi_2\Psi_1\p_{\xi_2}\hat{v}_1(\xi)\|_{L^2}\\
& \lesssim \big\|t^{-\frac{\mu}{2}}|\xi|\hat{v}_{0}(\xi)+t^{-\frac{\mu}{2}}t(|\xi_1|+|\xi_2|)|\xi|\hat{v}_{0}(\xi)
+t^{-\f{\mu}{2}}|\xi|(\xi_1\p_{\xi_1}+\xi_2\p_{\xi_2})\hat{v}_{0}(\xi)\big\|_{L^{2}}\\
&\quad +\big\|t^{-\frac{\mu}{2}}t(|\xi_1|+|\xi_2|)\hat{v}_{1}(\xi)\big\|_{L^{2}}
+\big\|t^{-\f{\mu}{2}}(\xi_1\p_{\xi_1}+\xi_2\p_{\xi_2})\hat{v}_{1}(\xi)\big\|_{L^{2}}\\
& \lesssim t^{-\frac{\mu}{2}}\left\|\nabla v_{0}\right\|_{Z, 1, 2}+ t^{-\f{\mu}{2}}\left\|t\p_1v_1\right\|_{L^2}+ t^{-\f{\mu}{2}}\left\|t\p_2v_1\right\|_{L^2}+ t^{-\f{\mu}{2}}\left\|(x_1\p_1+x_2\p_2)v_1\right\|_{L^2}+t^{-\f{\mu}{2}}\left\|v_1\right\|_{L^2}\\
&\lesssim  t^{-\frac{\mu}{2}}\left\|\nabla v_{0}\right\|_{Z, 1, 2}+ t^{-\f{\mu}{2}}\left\|v_1\right\|_{Z, 1,2}.
\end{aligned}
\end{equation}
For $|\xi|\leq1$ and $t|\xi|\geq1$,
by  \eqref{equ:q55}-\eqref{equ:q56} and \eqref{equ:q14}-\eqref{equ:q15}, one has that
\begin{equation}\label{equ:q84}
\begin{aligned}
&\left\||\xi|(\xi_1\p_{\xi_1}\hat{v}+\xi_2\p_{\xi_2})\hat{v}\right\|_{L^2(A_2)}
\lesssim \big\|t^{-\frac{\mu}{2}}|\xi|\hat{v}_{0}(\xi)+t^{-\frac{\mu}{2}}t(|\xi_1|+|\xi_2|)|\xi|\hat{v}_{0}(\xi)\big\|_{L^2}\\
&\quad +t^{-1}\big\|(t|\xi|)^{1-\frac{\mu}{2}}t(|\xi_1|+|\xi_2|)(\hat{v}_{0}(\xi)+\hat{v}_{1}(\xi))\big\|_{L^{2}}
+t^{-1}\big\|(t|\xi|)^{1-\f{\mu}{2}}(\xi_1\p_{\xi_1}+\xi_2\p_{\xi_2})(\hat{v}_{0}(\xi)+\hat{v}_{1}(\xi))\big\|_{L^{2}}\\
&\lesssim  t^{-\frac{\mu}{2}}\left\|\nabla v_{0}\right\|_{Z, 1, 2}+ t^{-1}\left\|(t\p_1+x_1\p_t)(v_0+v_1)\right\|_{L^2}+ t^{-1}\left\|(t\p_2+x_2\p_t)(v_0+v_1)\right\|_{L^2}\\
&\quad+ t^{-1}\left\|(x_1\p_1+x_2\p_2+t\p_t)(v_0+v_1)\right\|_{L^2}\\
&\lesssim t^{-\frac{\mu}{2}}\left\|\nabla v_{0}\right\|_{Z, 1, 2}+ t^{-1}\left\|v_0+v_1\right\|_{Z, 1, 2}.
\end{aligned}
\end{equation}
For  $t|\xi|\leq1$,  we have from \eqref{equ:q55}-\eqref{equ:q56} and \eqref{equ:q15} that
\begin{equation}\label{equ:q85}
\begin{aligned}
&\left\||\xi|(\xi_1\p_{\xi_1}\hat{v}+\xi_2\p_{\xi_2})\hat{v}\right\|_{L^2(A_3)}\\
&\lesssim \big\|t^{-\frac{\mu}{2}}|\xi|\hat{v}_{0}(\xi)\big\|_{L^{2}}+\big\||\xi|(\hat{v}_{0}(\xi)+\hat{v}_{1}(\xi))\big\|_{L^{2}}
+\left\||\xi|(\xi_1\p_{\xi_1}+\xi_2\p_{\xi_2})(\hat{v}_{0}(\xi)+\hat{v}_{1}(\xi))\right\|_{L^{2}}\\
&\lesssim  t^{-\frac{\mu}{2}}\left\|\nabla v_{0}\right\|_{Z, 1, 2}+ t^{-1}\left\|v_0+v_1\right\|_{L^2}+ t^{-1}\left\|(x_1\p_1+x_2\p_2)(v_0+v_1)\right\|_{L^2}\\
&\lesssim t^{-\frac{\mu}{2}}\left\|\nabla v_{0}\right\|_{Z, 1, 2}+ t^{-1}\left\|v_0+v_1\right\|_{Z, 1, 2}.
\end{aligned}
\end{equation}
Therefore, \eqref{equ:q79} follows from \eqref{equ:q82} and \eqref{equ:q83}-\eqref{equ:q85}.
\vskip 0.2 true cm

{\bf Case 3.  The treatment  of $\|L_j(\p v)(t, \cdot)\|_{L^2(\R^2)}(j=1, 2)$}

\vskip 0.2 true cm
We now establish
\begin{equation}\label{equ:q86}
\|L_j(\p v)(t, \cdot)\|_{L^2(\R^2)} \lesssim t^{-\f{\mu}{2}}\left\|\nabla v_0\right\|_{Z, 1, 2}
+t^{-\f{\mu}{2}}\left\|v_1\right\|_{Z, 1,2}+t^{-1}\left\|v_0+v_1\right\|_{Z, 1, 2}.
\end{equation}
It suffices only to estimate $\|L_1(\p v)(t, \cdot)\|_{L^2(\R^2)}$.
Note that $|x|\lesssim t$ holds for $(t, x)\in \operatorname{supp}v$. Then
\begin{equation}\label{equ:q87}
\begin{aligned}
\|L_1(\p v)(t, \cdot)\|_{L^2(\R^2)}
& \leq \|t|\xi_1||\xi|\hat{v}\|_{L^2(\R^2)}+\|t|\xi_1|\p_t\hat{v}\|_{L^2(\R^2)}+\|\f{x_1}{t}t\p_t\p v\|_{L^2(\R^2)}\\
& \lesssim \|t|\xi_1||\xi|\hat{v}\|_{L^2(\R^2)}+\|t|\xi_1|\p_t\hat{v}\|_{L^2(\R^2)}+\|t\p_t\p v\|_{L^2(\R^2)}.
\end{aligned}
\end{equation}

We next deal with each term in the right hand side of \eqref{equ:q87}.

By comparing the estimates of \eqref{equ:q20}, \eqref{equ:q28} and \eqref{equ:q26} with those of \eqref{equ:q32},
\eqref{equ:q36} and \eqref{equ:q39}, respectively, then $\|t|\xi_1||\xi|\hat{v}\|_{L^2(\R^2)}$ can be treated
as for $\|t|\xi_1|\p_t\hat{v}\|_{L^2(\R^2)}$ in \eqref{equ:q82}.

For the term $\|t\p_t\p v\|_{L^2(\R^2)}$, the following estimate has been obtained in the  proof procedure of
$\|L_0(\p v)(t, \cdot)\|_{L^2(\R^2)}$ in Case 2.
$$
\|t\p_t\p v\|_{L^2(\R^2)}\lesssim t^{-\f{\mu}{2}}\left\|\nabla v_0\right\|_{Z, 1, 2}+t^{-\f{\mu}{2}}\left\|v_1\right\|_{Z, 1,2}+t^{-1}\left\|v_0+v_1\right\|_{Z, 1, 2}.
$$
Finally, we estimate the term $\|t|\xi_1|\p_t\hat{v}\|_{L^2(\R^2)}$. Analogously to the treatment in \eqref{equ:q75},
one has that for $|\xi|\geq 1$,
\begin{equation}\label{equ:q88}
\begin{aligned}
\left\|t|\xi_1|\p_t\hat{v}\right\|_{L^2(A_1)} & \leq\left\|\p_t\Psi_{0 }(t, 1, \xi)t|\xi_1|\hat{v}_{1}\right\|_{L^{2}}+\left\|\p_t\Psi_{1 }(t, 1, \xi)t|\xi_1|\hat{v}_{1}\right\|_{L^{2}} \\
& \lesssim t^{-\frac{\mu}{2}}\left\|t|\xi_1||\xi| \hat{v}_{0}\right\|_{L^2}+t^{-\f{\mu}{2}} \|t|\xi_1|\hat{v}_{1}\|_{L^{2}}\\
& \lesssim t^{-\frac{\mu}{2}}\left\|\nabla v_{0}\right\|_{Z, 1, 2}+t^{-\f{\mu}{2}} \|v_{1}\|_{Z, 1, 2}.
\end{aligned}
\end{equation}
In addition, for $|\xi|\leq1$ and $t|\xi|\geq1$, it holds
\begin{equation}\label{equ:q89}
\left\|t|\xi_1|\p_t\hat{v}\right\|_{L^2(A_2)}
\lesssim t^{-1} \big\|(t|\xi|)^{1-\f{\mu}{2}}t|\xi_1|(\hat{v}_{0}+\hat{v}_{1})\big\|_{L^{2}} \lesssim t^{-1}\left\|v_0+v_1\right\|_{Z, 1, 2}.
\end{equation}
For $t|\xi|\leq1$,
$$
\left\|t|\xi_1|\p_t\hat{v}\right\|_{L^2(A_3)} \lesssim \left\||\xi|(\hat{v}_{0}+\hat{v}_{1})\right\|_{L^{2}}  \leq t^{-1}\left\|v_{0}+v_{1}\right\|_{L^{2}}.
$$
Therefore,
\begin{equation}\label{equ:q90}
\left\|t|\xi_1|\p_t\hat{v}\right\|_{L^2(\R^2)}\lesssim t^{-\f{\mu}{2}}\left\|\nabla v_0\right\|_{Z, 1, 2}
+t^{-\f{\mu}{2}}\left\| v_1\right\|_{Z, 1, 2}+t^{-1}\left\|v_0+v_1\right\|_{Z, 1, 2},
\end{equation}
which derives \eqref{equ:q86}.
\vskip 0.2 true cm

{\bf Case 4.  The treatment  of $\|\Omega_{12}(\p v)(t, \cdot)\|_{L^2(\R^2)}$}

\vskip 0.2 true cm
We will prove
\begin{equation}\label{equ:q91}
\|\Omega_{12}(\p v)(t, \cdot)\|_{L^2(\R^2)} \lesssim t^{-\f{\mu}{2}}\left\|\nabla v_0\right\|_{Z, 1, 2}
+t^{-\f{\mu}{2}}\left\| v_1\right\|_{Z, 1, 2}+t^{-1}\left\|v_0+v_1\right\|_{Z, 1, 2}.
\end{equation}

Analogously treated as in \eqref{equ:q71} and \eqref{equ:q87}, one can obtain
\begin{equation}\label{equ:q92}
\begin{aligned}
\|\Omega_{12}(\p v)(t, \cdot)\|_{L^2(\R^2)}
& \leq \|\f{x_1}{t}t\p_2\p v\|_{L^2(\R^2)}+\|\f{x_2}{t}t\p_1v\|_{L^2(\R^2)}\\
& \lesssim \|t|\xi_2||\xi|\hat{v}\|_{L^2(\R^2)}+\|t|\xi_1||\xi|\hat{v}\|_{L^2(\R^2)}\\
&\quad+\left\|t|\xi_2|\p_t\hat{v}\right\|_{L^2(\R^2)}+\left\|t|\xi_1|\p_t\hat{v}\right\|_{L^2(\R^2)}.
\end{aligned}
\end{equation}
In addition, by \eqref{equ:q90}, we have that for $j=1,2$,
$$
\left\|t|\xi_j|\p_t\hat{v}\right\|_{L^2(\R^2)} \lesssim t^{-\f{\mu}{2}}\left\|\nabla v_0\right\|_{Z, 1, 2}
+t^{-\f{\mu}{2}}\left\| v_1\right\|_{Z, 1, 2}+t^{-1}\left\|v_0+v_1\right\|_{Z, 1, 2}.
$$
On the other hand, as illustrated in \eqref{equ:q87}, it holds that for $j=1, 2$,
$$
\left\|t|\xi_j||\xi|\hat{v}\right\|_{L^2(\R^2)} \lesssim t^{-\f{\mu}{2}}\left\|\nabla v_0\right\|_{Z, 1, 2}
+t^{-\f{\mu}{2}}\left\| v_1\right\|_{Z, 1, 2}+t^{-1}\left\|v_0+v_1\right\|_{Z, 1, 2}.
$$
Therefore, \eqref{equ:q91} is shown.

Collecting all the cases above, we complete the proof of Lemma \ref{lem3}.

\end{proof}
\section{Estimates of solution to 2-D inhomogeneous equation $\square w+\f{\mu}{t}\,\p_tw=F$}\label{sec4}
In this section, we will derive some time-decay estimates of solution to the 2-D linear problem
\begin{equation}\label{equ:l1}
\left\{ \enspace
\begin{aligned}
&\square w+\f{\mu}{t}\,\p_tw=F, &&
t\geq 1,\\
&w(1,x)=0, \quad \partial_{t} w(1,x)=0, &&x\in\R^2,
\end{aligned}
\right.
\end{equation}
where $2<\mu<3$. To this end,  it follows from the  Duhamel's  principle that we first treat such a
homogeneous initial data problem from the initial variable time $\tau\ge 1$
\begin{equation}\label{equ:l2}
\left\{ \enspace
\begin{aligned}
&\square v+\f{\mu}{t}\,\p_tv=0, &&
t\geq \tau \geq1,\\
&v(\tau,x)=0, \quad \partial_{t} v(\tau,x)=v_1( x), &&x\in\R^2,
\end{aligned}
\right.
\end{equation}
where $v_1(x)\in C_0^{\infty}(\R^2)$.
\begin{lemma}\label{lem4}
Let $v$ solve \eqref{equ:l2}. Then for any $\varepsilon_1\in(0, 1)$, there exists a constant
$\delta(\varepsilon_1)=\f{2\varepsilon_1}{1+\varepsilon_1}>0$ such that
\begin{equation}\label{equ:l3}
\|v\|_{Z, 1, 2} \lesssim t^{\delta(\varepsilon_1)-1}\tau\left\|v_1\right\|_{Z, 1,(1+\varepsilon_1, 2)}
\end{equation}
and
\begin{equation}\label{equ:l4}
\|\p v\|_{Z, 1, 2} \lesssim t^{-1}\tau\left\|v_1\right\|_{Z, 1,2}.
\end{equation}
\end{lemma}
\begin{proof}
Although the proof procedure of Lemma \ref{lem4} is similar to that for  Lemma \ref{lem2} and Lemma \ref{lem3},
we still provide some details of proof for \eqref{equ:l3} and \eqref{equ:l4} due to  the appearance of the parameter $\tau$.

We now establish
$$
\|v\|_{Z, 1, 2} \lesssim t^{\delta(\varepsilon_1)-1}\tau\left\|v_1\right\|_{Z, 1,(1+\varepsilon_1, 2)}.
$$
Note that
\begin{equation}\label{Yson-5}
\|v\|_{Z, 1, 2} \leq \|\hat{v}\|_{L^2}+\sum_{j=1}^2\||\xi_j|\hat{v}\|_{L^2}+\|\p_t\hat{v}\|_{L^2}+\|\widehat{L_0v}\|_{L^2}
+\sum_{j=1}^2\|\widehat{L_jv}\|_{L^2}+\|\widehat{\Omega_{12}v}\|_{L^2}.
\end{equation}
Next we deal with each term in \eqref{Yson-5}.

It follows from \eqref{equ:q9} and \eqref{equ:q12} that for $t|\xi|\geq \tau|\xi|\geq 1$,
\begin{equation}\label{equ:l6}
 \left|\Psi_1(t, \tau, \xi)\right|  \lesssim t^\rho\tau^{1-\rho}(\tau|\xi|)^{-\frac{1}{2}}(t|\xi|)^{-\frac{1}{2}}=t^{-\frac{\mu}{2}}\tau^{\frac{\mu}{2}}|\xi|^{-1}.
\end{equation}
In addition, for $t|\xi|\geq 1$ and $ \tau|\xi|\leq 1$, by \eqref{equ:q11} and \eqref{equ:q14}-\eqref{equ:q15}, one has
\begin{equation}\label{equ:l7}
\left|\Psi_1(t, \tau, \xi)\right|  \lesssim t^\rho\tau^{1-\rho}(\tau|\xi|)^{\rho}(t|\xi|)^{-\frac{1}{2}}
=t^{-\frac{\mu}{2}}|\xi|^{-\frac{\mu}{2}}\tau.
\end{equation}
For $\tau|\xi|\leq t|\xi|\leq 1$, by \eqref{equ:q11} and \eqref{equ:q15}, we have
\begin{equation}\label{equ:l8}
\begin{aligned}
\left|\Psi_1(t, \tau, \xi)\right| & \lesssim t^\rho\tau^{1-\rho}\left[(\tau|\xi|)^{-\rho}(t|\xi|)^{\rho}
+(\tau|\xi|)^{\rho}(t|\xi|)^{-\rho}\right]\\
&=t^{2\rho}\tau^{1-2\rho}+\tau\lesssim \tau.
\end{aligned}
\end{equation}
Thus it follows from \eqref{equ:l6}-\eqref{equ:l8} and \eqref{equ:q4} with $\hat{v}_0 = 0$ that for $t\geq \tau\geq1$,
\begin{equation}\label{equ:l9}
\begin{aligned}
\|\hat{v}\|_{L^2} & \lesssim\big\|t^{-\frac{\mu}{2}} \tau^{\frac{\mu}{2}}|\xi|^{-1} \hat{v}_1(\xi)\big\|_{L^2(\tau|\xi| \geqslant 1)}+\big\|t^{-\frac{\mu}{2}}|\xi|^{-\frac{\mu}{2}} \tau \hat{v}_1(\xi)\big\|_{L^2(\tau|\xi| \leqslant 1, t|\xi| \geqslant 1)}+\left\|\tau \hat{v}_1(\xi)\right\|_{L^2(t|\xi| \leqslant 1)} \\
& \lesssim t^{-\frac{\mu}{2}} \tau^{\frac{\mu}{2}} \cdot \tau\left\|(1+\tau|\xi|)^{-1} \hat{v}_1\right\|_{L^2}+\tau\left\|(1+t|\xi|)^{-1} \hat{v}_1\right\|_{L^2}+\tau\left\|(1+\tau|\xi|)(1+\tau|\xi|)^{-1} \hat{v}_1\right\|_{L^2} \\
& \lesssim t^{1-\frac{2}{1+\varepsilon_1}} \tau\left\|v_1\right\|_{\left(1+\varepsilon_1, 2\right)}+\tau t^{1-\frac{2}{1+\varepsilon_1}}\left\|v_1\right\|_{\left(1+\varepsilon_1, 2\right)}+\tau \cdot \tau^{1-\frac{2}{1+\varepsilon_1}}\left\|v_1\right\|_{\left(1+\varepsilon_1, 2\right)} \\
& \lesssim t^{\delta(\varepsilon_1)-1} \tau\left\|v_1\right\|_{\left(1+\varepsilon_1, 2\right)}
\end{aligned}
\end{equation}
and
\begin{equation}\label{equ:l9-1}
\begin{aligned}
\|t|\xi_j|\hat{v}\|_{L^2} & \lesssim\big\|t^{-\frac{\mu}{2}} \tau^{\frac{\mu}{2}}|\xi|^{-1}t|\xi_j| \hat{v}_1\big\|_{L^2(\tau|\xi| \geqslant 1)}+\big\|t^{-\frac{\mu}{2}}|\xi|^{-\frac{\mu}{2}} \tau t|\xi_j| \hat{v}_1\big\|_{L^2(\tau|\xi| \leqslant 1, t|\xi| \geqslant 1)}+\left\|\tau t|\xi_j|\hat{v}_1\right\|_{L^2(t|\xi| \leqslant 1)} \\
& \lesssim t^{-\frac{\mu}{2}} \tau^{\frac{\mu}{2}+1} \left\|(1+\tau|\xi|)^{-1} t|\xi_j|\hat{v}_1\right\|_{L^2}+\tau\left\|(1+t|\xi|)^{-1} t|\xi_j|\hat{v}_1\right\|_{L^2}+\tau\left\|(1+\tau|\xi|)^{-1} t|\xi_j|\hat{v}_1\right\|_{L^2} \\
& \lesssim  \tau \cdot \tau^{1-\frac{2}{1+\varepsilon_1}}\left\|v_1\right\|_{Z, 1, \left(1+\varepsilon_1, 2\right)}+\tau \cdot t^{1-\frac{2}{1+\varepsilon_1}}\left\|v_1\right\|_{Z, 1, \left(1+\varepsilon_1, 2\right)}+\tau \cdot \tau^{1-\frac{2}{1+\varepsilon_1}}\left\|v_1\right\|_{Z, 1, \left(1+\varepsilon_1, 2\right)} \\
& \lesssim  t^{\delta(\varepsilon_1)-1} \tau\left\|v_1\right\|_{Z, 1, \left(1+\varepsilon_1, 2\right)}.
\end{aligned}
\end{equation}
Then
\begin{equation}\label{equ:l9-2}
\||\xi_j|\hat{v}\|_{L^2}=t^{-1}\|t|\xi_j|\p_t\hat{v}\|_{L^2}
 \lesssim t^{-\frac{2}{1+\varepsilon_1}} \tau\left\|v_1\right\|_{Z, 1, \left(1+\varepsilon_1, 2\right)}\leq t^{\delta(\varepsilon_1)-1} \tau\left\|v_1\right\|_{Z, 1, \left(1+\varepsilon_1, 2\right)}.
\end{equation}
In addition, for $t|\xi|\geq\tau|\xi|\geq 1$,  we have from \eqref{equ:q9} and \eqref{equ:q12} that
\begin{equation}\label{equ:l10}
 \left|\p_t\Psi_1(t, \tau, \xi)\right|  \lesssim |\xi| t^\rho\tau^{1-\rho}(\tau|\xi|)^{-\frac{1}{2}}(t|\xi|)^{-\frac{1}{2}}=t^{-\frac{\mu}{2}}\tau^{\frac{\mu}{2}}.
\end{equation}
For $t|\xi|\geq 1$ and $ \tau|\xi|\leq 1$, by \eqref{equ:q11} and \eqref{equ:q14}-\eqref{equ:q15}, one has
\begin{equation}\label{equ:l11}
\left|\p_t\Psi_1(t, \tau, \xi)\right|  \lesssim |\xi| t^\rho\tau^{1-\rho}(\tau|\xi|)^{-\rho}(t|\xi|)^{-\frac{1}{2}}=t^{-\frac{\mu}{2}}|\xi|^{\frac{\mu}{2}}\tau^{\mu}.
\end{equation}
For $\tau|\xi|\leq t|\xi|\leq 1$, by \eqref{equ:q11} and \eqref{equ:q15}, it holds
\begin{equation}\label{equ:l12}
\begin{aligned}
\left|\p_t\Psi_1(t, \tau, \xi)\right| & \lesssim |\xi| t^\rho\tau^{1-\rho}\left[(\tau|\xi|)^{-\rho}(t|\xi|)^{\rho-1}+(\tau|\xi|)^{\rho}(t|\xi|)^{1-\rho}\right]\\
& \leq t^{-1}\tau.
\end{aligned}
\end{equation}
Hence we have
\begin{equation}\label{equ:l13}
\begin{aligned}
\|t\p_t\hat{v}\|_{L^2} & \lesssim\big\|t^{1-\frac{\mu}{2}} \tau^{\frac{\mu}{2}} \hat{v}_1(\xi)\big\|_{L^2(\tau|\xi| \geqslant 1)}+\big\|t^{1-\frac{\mu}{2}}|\xi|^{\frac{\mu}{2}} \tau^\mu \hat{v}_1(\xi)\big\|_{L^2(\tau|\xi| \leqslant 1, t|\xi| \geqslant 1)}+\left\|\tau \hat{v}_1(\xi)\right\|_{L^2(t|\xi| \leqslant 1)} \\
& \lesssim t^{-\frac{\mu}{2}} \tau^{\frac{\mu}{2}+1} \left\|(1+\tau|\xi|)^{-1} t|\xi| \hat{v}_1\right\|_{L^2}+t^{1-\f{\mu}{2}}\tau^{\f{\mu}{2}-1}\big\|\tau^{\f{\mu}{2}+1}|\xi|^{\f{\mu}{2}}(1+t|\xi|)^{-1} t|\xi|\hat{v}_1\big\|_{L^2(\tau|\xi| \leqslant 1, t|\xi| \geqslant 1)}\\
&\quad+\tau\left\|(1+\tau|\xi|)(1+\tau|\xi|)^{-1} \hat{v}_1\right\|_{L^2} \\
& \lesssim \tau \tau^{1-\frac{2}{1+\varepsilon_1}} \left\|t\p_1v_1\right\|_{\left(1+\varepsilon_1, 2\right)}+\tau \tau^{1-\frac{2}{1+\varepsilon_1}} \left\|t\p_2v_1\right\|_{\left(1+\varepsilon_1, 2\right)}+\tau t^{1-\frac{2}{1+\varepsilon_1}}\left\|t\p_1v_1\right\|_{\left(1+\varepsilon_1, 2\right)}\\
&\quad+\tau t^{1-\frac{2}{1+\varepsilon_1}}\left\|t\p_2v_1\right\|_{\left(1+\varepsilon_1, 2\right)}+\tau \cdot \tau^{1-\frac{2}{1+\varepsilon_1}}\left\|v_1\right\|_{\left(1+\varepsilon_1, 2\right)} \\
& \lesssim t^{\delta(\varepsilon_1)-1} \tau\left\|v_1\right\|_{Z, 1, \left(1+\varepsilon_1, 2\right)}
\end{aligned}
\end{equation}
and
\begin{equation}\label{equ:l14}
\|\p_t\hat{v}\|_{L^2}=t^{-1}\|t\p_t\hat{v}\|_{L^2}
 \lesssim t^{\delta(\varepsilon_1)-1} \tau\left\|v_1\right\|_{Z, 1, \left(1+\varepsilon_1, 2\right)}.
\end{equation}
By an analogous analysis in the proof process of \eqref{equ:q50}, one has that
for $j=1,2$, and $\tau|\xi|\geq1$,
\begin{equation}\label{equ:l15}
\begin{aligned}
\partial_{\xi_j}&\Psi_1(t, \tau, \xi)=\frac{i \pi}{4} t^\rho \tau^{1-\rho}\left[\frac{\tau\xi_j}{|\xi|}\big(H_{\rho-1}^{-}(\tau|\xi|)
-\frac{\rho}{\tau|\xi|} H_\rho^{-}(\tau|\xi|)\big) H_\rho^{+}(t|\xi|)+\frac{t\xi_j}{|\xi|}\left(H_{\rho-1}^{+}(t|\xi|)\right.\right. \\
& \quad-\left.\frac{\rho}{t|\xi|} H_\rho^{+}(t|\xi|)\right) H_\rho^{-}(\tau|\xi|)-\frac{\tau\xi_j}{|\xi|}\big(H_{\rho-1}^{+}(\tau|\xi|)-\frac{\rho}{\tau|\xi|} H_\rho^{+}(\tau|\xi|)\big) H_\rho^{-}(t|\xi|)\\
&\quad\left.-\frac{t\xi_j}{|\xi|}\big(H_{\rho-1}^{-}(t|\xi|)-\frac{\rho}{t|\xi|} H_\rho^{-}(t|\xi|)\big) H_\rho^{+}(\tau|\xi|)\right]\\
& =\frac{i \pi}{4}t^\rho \tau^{1-\rho}\left[\frac{\tau\xi_j}{|\xi|} H_{\rho-1}^{-}(\tau|\xi|) H_\rho^{+}(t|\xi|)+\frac{t\xi_j}{|\xi|} H_{\rho-1}^{+}(t|\xi|) H_\rho^{-}(\tau|\xi|)\right. \\
&\quad -\frac{\tau\xi_j}{|\xi|} H_{\rho-1}^{+}(\tau|\xi|) H_\rho^{-}(t|\xi|)
-\frac{t\xi_j}{|\xi|} H_{\rho-1}^{-}(t|\xi|) H_\rho^{+}(\tau|\xi|)-\frac{\xi_j\rho}{|\xi|^2} H_\rho^{-}(\tau|\xi|) H_\rho^{+}(t|\xi|)\\
& \left.\quad +\frac{\xi_j\rho}{|\xi|^2} H_\rho^{+}(t|\xi|) H_\rho^{-}(\tau|\xi|) +\frac{\xi_j\rho}{|\xi|^2} H_\rho^{+}(\tau|\xi|) H_\rho^{-}(t|\xi|)+\frac{\xi_j\rho}{|\xi|^2} H_\rho^{-}(t|\xi|) H_\rho^{+}(\tau|\xi|)\right]
\end{aligned}
\end{equation}
and for $\tau|\xi|\leq 1$,
\begin{equation}\label{equ:l17}
\begin{aligned}
\partial_{\xi_j}& \Psi_1(t, \tau, \xi)=\frac{ \pi}{2}\csc(\rho\pi) t^\rho \tau^{1-\rho}\left[\frac{\tau\xi_j}{|\xi|}\big(J_{-\rho-1}(\tau|\xi|)+\frac{\rho}{\tau|\xi|} J_{-\rho}(\tau|\xi|)\big) J_\rho(t|\xi|)+\frac{t\xi_j}{|\xi|}\bigg(J_{\rho-1}(t|\xi|)\right. \\
& \quad-\frac{\rho}{t|\xi|} J_\rho(t|\xi|)\bigg) J_{-\rho}(\tau|\xi|)-\frac{\tau\xi_j}{|\xi|}\big(J_{\rho-1}(\tau|\xi|)-\frac{\rho}{\tau|\xi|} J_\rho(\tau|\xi|)\big) J_{-\rho}(t|\xi|) \\
& \quad\left. -\frac{t\xi_j}{|\xi|}\big(J_{-\rho-1}(t|\xi|)+\frac{\rho}{t|\xi|} J_\rho(t|\xi|)\big) J_\rho(\tau|\xi|)\right]  \\
& =\frac{\pi}{2}\csc(\rho\pi)t^\rho \tau^{1-\rho}\left[\frac{\tau\xi_j}{|\xi|} J_{-\rho-1}(\tau|\xi|) J_\rho(t|\xi|)+\frac{t\xi_j}{|\xi|} J_{\rho-1}(t|\xi|) J_{-\rho}(\tau|\xi|)\right. \\
&\quad -\frac{\tau\xi_j}{|\xi|} J_{\rho-1}(\tau|\xi|) J_{-\rho}(t|\xi|)
-\frac{t\xi_j}{|\xi|} J_{-\rho-1}(t|\xi|) J_\rho(\tau|\xi|)+\frac{\xi_j\rho}{|\xi|^2} J_{-\rho}(\tau|\xi|) J_\rho(t|\xi|)\\
& \left.\quad -\frac{\xi_j\rho}{|\xi|^2} J_\rho(t|\xi|) J_{-\rho}(\tau|\xi|) +\frac{\xi_j\rho}{|\xi|^2} J_\rho(\tau|\xi|) J_{-\rho}(t|\xi|)-\frac{\xi_j\rho}{|\xi|^2} J_{-\rho}(t|\xi|) J_\rho(\tau|\xi|)\right].
\end{aligned}
\end{equation}
As in \eqref{equ:q61} with $\hat{v}_0=0$, it follows from \eqref{equ:l15} that
\begin{equation}\label{equ:l19}
\begin{aligned}
&\left\|\xi_1\p_{\xi_1}\hat{v}+\xi_2\p_{\xi_2}\hat{v}\right\|_{L^2(\tau|\xi|\geq1)}\\
&=\|\xi_1(\p_{\xi_1}\Psi_1)\hat{v_1}(\xi)+\xi_1\Psi_1\p_{\xi_1}\hat{v_1}(\xi)
+\xi_2(\p_{\xi_2}\Psi_1)\hat{v_1}(\xi)+\xi_2\Psi_1\p_{\xi_2}\hat{v_1}(\xi)\|_{L^2(\tau|\xi|\geq1)}\\
& \lesssim\big\|t^{1-\frac{\mu}{2}} \tau^{\frac{\mu}{2}} \hat{v}_1(\xi)\big\|_{L^2(\tau|\xi| \geqslant 1)}
+\big\|t^{-\frac{\mu}{2}} \tau^{\frac{\mu}{2}}|\xi|^{-1}\left(\xi_1 \partial_{ \xi_1}+\xi_2\partial_ {\xi_2}\right) \hat{v}_1\left(\xi\right)\big\|_{L^2(\tau|\xi| \geqslant 1)}\\
& \lesssim t^{-\frac{\mu}{2}} \tau^{\frac{\mu}{2}+1} \left\|(1+\tau|\xi|)^{-1} t|\xi| \hat{v}_1\right\|_{L^2}+\tau\left\|(1+\tau|\xi|)^{-1}\left(\xi_1 \partial_{ \xi_1}+\xi_2 \partial_ {\xi_2}\right) \hat{v}_1\left(\xi\right)\right\|_{L^2}\\
&\lesssim  \tau \tau^{1-\frac{2}{1+\varepsilon_1}} \left\|t\p_1v_1\right\|_{\left(1+\varepsilon_1, 2\right)}+\tau \tau^{1-\frac{2}{1+\varepsilon_1}} \left\|t\p_2v_1\right\|_{\left(1+\varepsilon_1, 2\right)}+\tau \tau^{1-\frac{2}{1+\varepsilon_1}} \left\|(x_1\p_1+x_2\p_2)v_1\right\|_{\left(1+\varepsilon_1, 2\right)}\\
&\quad+\tau \tau^{1-\frac{2}{1+\varepsilon_1}} \left\|v_1\right\|_{\left(1+\varepsilon_1, 2\right)}\\
&\lesssim t^{1-\frac{2}{1+\varepsilon_1}} \tau\left\|v_1\right\|_{Z, 1, \left(1+\varepsilon_1, 2\right)}.
\end{aligned}
\end{equation}
Similarly, by \eqref{equ:l17}, one can easily obtain
\begin{equation}\label{equ:l20-1}
\left\|\xi_1\p_{\xi_1}\hat{v}+\xi_2\p_{\xi_2}\hat{v}\right\|_{L^2(\tau|\xi|\leq1)}
\lesssim t^{1-\frac{2}{1+\varepsilon_1}} \tau\left\|v_1\right\|_{Z, 1, \left(1+\varepsilon_1, 2\right)}.
\end{equation}
Therefore, by \eqref{equ:l9}, \eqref{equ:l13} and \eqref{equ:l19}-\eqref{equ:l20-1}, we arrive at
$$
\|\widehat{L_0v}\|_{L^2}\lesssim 2\|\hat{v}\|_{L^2(\R^2)}+\|t\p_t\hat{v}\|_{L^2(\R^2)}+\|\xi_1\p_{\xi_1}\hat{v}+\xi_2\p_{\xi_2}\hat{v}\|_{L^2(\R^2)}\lesssim t^{1-\frac{2}{1+\varepsilon_1}} \tau\left\|v_1\right\|_{Z, 1, \left(1+\varepsilon_1, 2\right)}.
$$
Meanwhile, by \eqref{equ:l9-2} and \eqref{equ:l14}, one has
$$
\|\widehat{\p v}\|_{L^2}\lesssim \sum_{j=1}^2\||\xi_j|\hat{v}\|_{L^2}+\|\p_t\hat{v}\|_{L^2}\lesssim t^{1-\frac{2}{1+\varepsilon_1}} \tau\left\|v_1\right\|_{Z, 1, \left(1+\varepsilon_1, 2\right)}.
$$
In addition, as in \eqref{equ:q63} and \eqref{equ:q71}, it follows from \eqref{equ:l9-1} and \eqref{equ:l13} that
$$
\sum_{j=1}^2\|\widehat{L_jv}\|_{L^2}\lesssim \|t|\xi_1|\hat{v}\|_{L^2(\R^2)}+\|t|\xi_2|\hat{v}\|_{L^2(\R^2)}+\|t\p_t\hat{v}\|_{L^2(\R^2)}\lesssim t^{1-\frac{2}{1+\varepsilon_1}} \tau\left\|v_1\right\|_{Z, 1, \left(1+\varepsilon_1, 2\right)}
$$
and
$$
\|\widehat{\Omega_{12}v}\|_{L^2}\lesssim \left\|t|\xi_2|\p_t\hat{v}\right\|_{L^2(\R^2)}+\left\|t|\xi_1|\p_t\hat{v}\right\|_{L^2(\R^2)}\lesssim t^{1-\frac{2}{1+\varepsilon_1}} \tau\left\|v_1\right\|_{Z, 1, \left(1+\varepsilon_1, 2\right)}.
$$
Collecting all the results above yields \eqref{equ:l3}.

Next we estimate  $\|\p v\|_{Z, 1, 2}$. For this purpose, by an analogous analysis in the proof procedure
of \eqref{equ:q74}, it suffices only to treat $\|t|\xi|\p_t\hat{v}\|_{L^2}$, $\||\xi_j||\xi|\hat{v}\|_{L^2}(j=1, 2)$ and $\||\xi|(\xi_1\p_{\xi_1}\hat{v}+\xi_2\p_{\xi_2})\hat{v}\|_{L^2(\R^2)}$.

It follows from \eqref{equ:l10}-\eqref{equ:l12} that
\begin{equation}\label{equ:l20}
\begin{aligned}
\|t|\xi|\p_t\hat{v}\|_{L^2} & \lesssim\big\|t^{1-\frac{\mu}{2}} \tau^{\frac{\mu}{2}} |\xi|\hat{v}_1(\xi)\big\|_{L^2(\tau|\xi| \geqslant 1)}+\big\|t^{1-\frac{\mu}{2}}|\xi|^{\frac{\mu}{2}} \tau^\mu |\xi| \hat{v}_1(\xi)\big\|_{L^2(\tau|\xi| \leqslant 1, t|\xi| \geqslant 1)}
+\left\|\tau |\xi|\hat{v}_1(\xi)\right\|_{L^2(t|\xi| \leqslant 1)} \\
& \lesssim t^{-1} \tau \big\|t^{1-\f{\mu}{2}}\tau^{\f{\mu}{2}-1} t|\xi| \hat{v}_1\big\|_{L^2(\tau|\xi| \geqslant 1)}
+t^{-1} \tau \big\|t^{1-\f{\mu}{2}}\tau^{\f{\mu}{2}-1}(\tau|\xi|)^{\f{\mu}{2}} t|\xi| \hat{v}_1\big\|_{L^2(\tau|\xi| \leqslant 1, t|\xi| \geqslant 1)}\\
&\quad+t^{-1}\left\|\tau \hat{v}_1(\xi)\right\|_{L^2(t|\xi| \leqslant 1)} \\
& \lesssim t^{-1}\tau\|v_1\|_{Z, 1, 2}.
\end{aligned}
\end{equation}
By \eqref{equ:l6}-\eqref{equ:l8}, one has that for $j=1, 2$,
\begin{equation}\label{equ:l21}
\begin{aligned}
\||\xi||\xi_j|\hat{v}\|_{L^2}&\lesssim\big\|t^{-\frac{\mu}{2}} \tau^{\frac{\mu}{2}}|\xi|^{-1}|\xi||\xi_j| \hat{v}_1\big\|_{L^2(\tau|\xi| \geqslant 1)}+\big\|t^{-\frac{\mu}{2}}|\xi|^{-\frac{\mu}{2}} \tau |\xi||\xi_j| \hat{v}_1\big\|_{L^2(\tau|\xi| \leqslant 1, t|\xi| \geqslant 1)}\\
&\quad +\left\|\tau |\xi||\xi_j|\hat{v}_1\right\|_{L^2(t|\xi| \leqslant 1)} \\
& \lesssim t^{-1} \tau \big\|t^{1-\f{\mu}{2}}\tau^{\f{\mu}{2}-1} |\xi_j| \hat{v}_1\big\|_{L^2(\tau|\xi| \geqslant 1)}+t^{-1} \tau \big\|(t|\xi|)^{1-\f{\mu}{2}} |\xi_j| \hat{v}_1\big\|_{L^2(\tau|\xi| \leqslant 1, t|\xi| \geqslant 1)}\\
&\quad+t^{-1}\tau\left\| |\xi_j|\hat{v}_1(\xi)\right\|_{L^2(t|\xi| \leqslant 1)} \\
& \lesssim t^{-1}\tau\|v_1\|_{Z, 1, 2}.
\end{aligned}
\end{equation}
As in \eqref{equ:l19}, by \eqref{equ:l15}, one obtains
\begin{equation}\label{equ:l22}
\begin{aligned}
&\||\xi|(\xi_1\p_{\xi_1}\hat{v}+\xi_2\p_{\xi_2})\hat{v}\|_{L^2(\R^2)}
\\
& \lesssim\big\|t^{1-\frac{\mu}{2}} \tau^{\frac{\mu}{2}}|\xi| \hat{v}_1(\xi)\big\|_{L^2(\tau|\xi| \geqslant 1)}
+\|t^{-\frac{\mu}{2}} \tau^{\frac{\mu}{2}}\left(\xi_1 \partial_{ \xi_1}+\xi_2 \partial_ {\xi_2}\right) \hat{v}_1\left(\xi\right)\|_{L^2(\tau|\xi| \geqslant 1)}\\
&\quad + \|t^{-\f{\mu}{2}+1}\tau|\xi|^{-\f{\mu}{2}+2}\hat{v}_1(\xi)\|_{L^2(\tau|\xi| \leq 1, t|\xi|\geq1)}+t^{-\f{\mu}{2}}\tau\|(\xi_1\p_{\xi_1}+\xi_2\p_{\xi_2})\hat{v}_1(\xi)\|_{L^2(\tau|\xi| \leq 1, t|\xi|\geq1)}\\
&\quad+
\|\tau|\xi|\hat{v}_1(\xi)\|_{L^2(t|\xi| \leq 1)}+t^{-1}\tau\|(\xi_1\p_{\xi_1}+\xi_2\p_{\xi_2})\hat{v}_1(\xi)\|_{L^2(t|\xi|\leq1)}\\
& \lesssim t^{-1} \tau \left\| t(|\xi_1|+|\xi_2|) \hat{v}_1\right\|_{L^2}+t^{-1}\tau\left\|\left(\xi_1 \partial_{ \xi_1}+\xi_2\partial_ {\xi_2}\right) \hat{v}_1\right\|_{L^2}\\
& \lesssim t^{-1}\tau\|v_1\|_{Z, 1, 2}.
\end{aligned}
\end{equation}
Therefore, in light of \eqref{equ:l20}-\eqref{equ:l22}, then \eqref{equ:l4} and further Lemma \ref{lem4} are proved.
\end{proof}

\section{Proof of Theorem~\ref{YH-1}}
Based on Lemmas \ref{lem2}-\ref{lem3} and Lemma \ref{lem4}, as in \cite{LX} or $\S 4$ of \cite{Rei1},
for any $T>1$, one can introduce the function space $X(T)$ with norm
\begin{equation}\label{L1}
\|u\|_{X(T)}:=\sup _{t \in[1, T]}\big(t^{-(\delta-1)}\|u\|_{Z, 1, 2}+t\|\p u\|_{Z, 1, 2}\big),
\end{equation}
where $\delta=\delta(\varepsilon_1)$ is defined by Lemma \ref{lem4}.

Based on Duhamel's principle and the  expression \eqref{equ:q4}, we define the nonlinear mapping $\mathcal{N}$ by
\begin{equation}\label{L2}
\begin{aligned}
\mathcal{N} u(t, x)&=\ve\Psi_{0}(t, 1, D)u_{0}(x)+\ve\Psi_{1}(t, 1, D) u_{1}(x)
+\int_{1}^{t} \Psi_{1}(t, \tau, D)|u(\tau, x)|^{p} \mathrm{d} \tau  \\
&:=\ve u^{lin}(t, x)+u^{non}(t, x)
\end{aligned}
\end{equation}
and introduce the following closed subset of $X(T)$
$$
X(T, M)=\left\{u \in X(T):\|u\|_{X(T)} \leq M\ve\right\},
$$
where $M$ is a fixed positive constant to be determined (see \eqref{L6} below). We now have
\begin{proposition}\label{prop-1}
For any $\varepsilon_1\in(0,1)$, $T>1$ and  $u, v \in X(T)$, the following estimates hold:
\begin{equation}\label{L3}
\|\mathcal{N} u\|_{X(T)} \leq C\big(\ve\left\|\left(u_{0}, u_{1}\right)\right\|_{\mathcal{A}}+\|u\|_{X(T)}^{p}\big)
\end{equation}
and
\begin{equation}\label{L4}
\|\mathcal{N} u-\mathcal{N} v\|_{X(T)} \leq C\|u-v\|_{X(T)}\big(\|u\|_{X(T)}^{p-1}+\|v\|_{X(T)}^{p-1}\big),
\end{equation}
where and below the constant $C>0$ is independent of $\ve>0$ and $T>1$, and
\begin{equation}\label{L5}
\begin{aligned}
\left\|\left(u_{0}, u_{1}\right)\right\|_{\mathcal{A}}=&\|u_0\|_{Z, 1, 2}+\|\nabla u_0\|_{Z, 1, 2}
+\|u_1\|_{Z, 1, 2}+\|u_1\|_{Z, 1, (1+\varepsilon_1, 2)}\\
&+\|u_0+u_1\|_{Z, 1, (1+\varepsilon_1, 2)}+\|u_0+u_1\|_{Z, 1, 2}.
\end{aligned}
\end{equation}
\end{proposition}

\begin{proof}
By \eqref{equ:q17} and \eqref{equ:l3}, one has
\begin{equation}\label{L8}
\begin{aligned}
\|\mathcal{N} u(t, \cdot)\|_{Z, 1, 2} &\lesssim t^{-\f{\mu}{2}}\ve\left\|u_0\right\|_{Z, 1, 2}
+t^{\delta(\varepsilon_1)-1}\ve\left\|u_1\right\|_{Z, 1,(1+\varepsilon_1, 2)}
+t^{\delta(\varepsilon_1)-1}\ve\left\|u_0+u_1\right\|_{Z, 1,(1+\varepsilon_1, 2)}\\
&\quad+ t^{\delta(\varepsilon_1)-1}\int_{1}^{t} \tau\left\||u(\tau, \cdot)|^p\right\|_{Z, 1,(1+\varepsilon_1, 2)}\mathrm{d}\tau,
\end{aligned}
\end{equation}
where $\delta=\delta(\varepsilon_1)$ is defined by Lemma \ref{lem4}. It follows from Lemma 2.1 in \cite{LX} that
the following H\"{o}lder's inequality holds
\begin{equation}\label{L9}
\|fg\|_{(p,q)}\lesssim\|f\|_{(p_1, q_1)}\|g\|_{(p_2, q_2)},
\end{equation}
where $\f{1}{p}=\f{1}{p_1}+\f{1}{p_2}\leq 1$ and $\f{1}{q}=\f{1}{q_1}+\f{1}{q_2}\leq 1$.
Then it follows from \eqref{L9} and  direct computation that
\begin{equation}\label{L13}
\left\||u(\tau, \cdot)|^p\right\|_{Z, 1,(1+\varepsilon_1, 2)} \lesssim\left\||u|^{p-1}\right\|_{(\tilde{q}, 2)}\|u(\tau, \cdot)\|_{Z, 1,2}\leq \left\||u|^{p-1}\right\|_{(\tilde{q}, \infty)}\|u(\tau, \cdot)\|_{Z, 1,2},
\end{equation}
where $\tilde{q}=\tilde{q}(\varepsilon_1)\in(2, \infty)$ satisfies
\begin{equation}\label{L14-1}
\f{1}{1+\varepsilon_1}=\f{1}{2}+\f{1}{\tilde{q}}.
\end{equation}
Denote  $\kappa=\kappa(\varepsilon_1)=\f{2}{\tilde{q}}=\f{1-\varepsilon_1}{1+\varepsilon_1}<1$.
Then, by $p>2>1+\kappa$ and H\"{o}lder's inequality \eqref{L9}, we can arrive at
\begin{equation}\label{L10}
\left\||u(\tau, \cdot)|^{p-1}\right\|_{(\tilde{q}, \infty)} \lesssim\|u(\tau, \cdot)\|_{\infty}^{p-1-\kappa}\|u(\tau, \cdot)\|_{(2, \infty)}^\kappa.
\end{equation}
In addition, the Sobolev embedding theorem on $S^1$ implies
\begin{equation}\label{L11}
\begin{aligned}
& \|u(\tau, \cdot)\|_{(2, \infty)}^2 \lesssim \int_0^{\infty}\left\|u(r\xi)\right\|_{H^{1}\left(S^{1}\right)}^2 r \mathrm{d} r \\
& \leq \int_{\mathbb{R}^2} u^2 \mathrm{d}x_1 \mathrm{d}x_2  +\int_0^{\infty}\int_0^{2 \pi}[\p_ {x_1} u(rcos\theta, rsin\theta) \cdot(-rsin\theta) +\p_{x_2} u(rcos\theta,rsin\theta)\cdot(rcos\theta)]^2r \mathrm{d} r \mathrm{d}\theta \\
& =\int_{\mathbb{R}^2} u^2 \mathrm{d} x_1 \mathrm{d} x_2  +\int_{\mathbb{R}^2}(x_1\p_{x_2}u-x_2\p_{x_1}u)^2\mathrm{d} x_1 \mathrm{d} x_2 \\
& \lesssim \|u\|_{Z, 1, 2}^2.
\end{aligned}
\end{equation}
On the other hand, it follows from the Klainerman-Sobolev inequality in \cite{KS} that
\begin{equation}\label{L12}
\|u(\tau, \cdot)\|_{\infty} \lesssim \tau^{\frac{1}{2}}\left(\|u(\tau, \cdot)\|_{Z, 1,2}+\|\p u(\tau, \cdot)\|_{Z, 1,2}\right).
\end{equation}
Substituting \eqref{L11} and \eqref{L12} into \eqref{L10} yields
$$
\left\||u(\tau, \cdot)|^{p-1}\right\|_{(\tilde{q}, \infty)} \lesssim
\tau^{\frac{p-1-\kappa}{2}}\left(\|u(\tau, \cdot)\|_{Z, 1,2}+\|\p  u(\tau, \cdot)\|_{Z, 1,2}\right)^{p-1-\kappa}\|u\|_{Z, 1, 2}^\kappa.
$$
By the definition of norm \eqref{L1} in space $X(T)$, it holds that
$$
\left\||u(\tau, \cdot)|^{p-1}\right\|_{(\tilde{q}, \infty)}\lesssim \tau^{\frac{p-1-\kappa}{2}+(p-1-\kappa)(\delta-1)+\kappa(\delta-1)}\|u\|_{X(T)}^{p-1}.
$$
Substituting it into \eqref{L13} and then \eqref{L8} implies
\begin{equation}\label{L14}
\begin{aligned}
\|\mathcal{N} u(t, \cdot)\|_{Z, 1, 2} &\lesssim t^{-\f{\mu}{2}}\ve\left\|u_0\right\|_{Z, 1, 2}
+t^{\delta-1}\ve\left\|u_1\right\|_{Z, 1,(1+\varepsilon_1, 2)}
+t^{\delta-1}\ve\left\|u_0+u_1\right\|_{Z, 1,(1+\varepsilon_1, 2)}\\
&\quad+ t^{\delta-1}\|u\|_{X(T)}^p\int_{1}^{t} \tau^{1+\frac{p-1-\kappa}{2}+(p-1-\kappa)(\delta-1)
+(\kappa+1)(\delta-1)}\mathrm{d}\tau,
\end{aligned}
\end{equation}
Note that $\delta=\f{2\varepsilon_1}{1+\varepsilon_1}\rightarrow 0+$ and $\kappa=\f{1-\varepsilon_1}{1+\varepsilon_1}\rightarrow 1-$ as $\varepsilon_1\rightarrow0+$. Therefore, for $p>2$, one can choose a sufficiently small $\varepsilon_1>0$ such that
$$
1+\frac{p-1-\kappa}{2}+p(\delta-1)<-1.
$$
Then
\begin{equation}\label{L15}
\begin{aligned}
t^{-(\delta-1)}\|\mathcal{N} u(t, \cdot)\|_{Z, 1, 2} &\lesssim \ve\left\|u_0\right\|_{Z, 1, 2}
+\ve\left\|u_1\right\|_{Z, 1,(1+\varepsilon_1, 2)}+\ve\left\|u_0+u_1\right\|_{Z, 1,(1+\varepsilon_1, 2)}+ \|u\|_{X(T)}^p\\
& \lesssim \ve\left\|\left(u_{0}, u_{1}\right)\right\|_{\mathcal{A}}+ \|u\|_{X(T)}^p.
\end{aligned}
\end{equation}
Similarly, by \eqref{equ:q74} and \eqref{equ:l4}, we have
\begin{equation}\label{L16}
\begin{aligned}
\|\p\mathcal{N}u(t, \cdot)\|_{Z, 1, 2} &\lesssim t^{-\f{\mu}{2}}\ve\left\|\nabla u_0\right\|_{Z, 1, 2}
+t^{-\f{\mu}{2}}\ve\left\| u_1\right\|_{Z, 1, 2}+t^{-1}\ve\left\|u_0+u_1\right\|_{Z, 1,2}\\
&\quad+ t^{-1}\int_{1}^{t} \tau\left\||u(\tau, \cdot)|^p\right\|_{Z, 1,2}\mathrm{d}\tau.
\end{aligned}
\end{equation}
As treated in \eqref{L13}, one has
\begin{equation}\label{L17}
\left\||u(\tau, \cdot)|^p\right\|_{Z, 1,2} \lesssim\left\||u(\tau, \cdot)|^{p-1}\right\|_{2+\varepsilon_2}\|u(\tau, \cdot)\|_{Z, 1, \bar{q}},
\end{equation}
where $\bar{q}=\bar{q}(\varepsilon_2)\in(2, \infty)$ fulfills
\begin{equation}\label{L17-1}
\f{1}{2}=\f{1}{2+\varepsilon_2}+\f{1}{\bar{q}}.
\end{equation}
In addition, it follows from Sobolev embedding theorem that
\begin{equation}\label{L18}
\|u(\tau, \cdot)\|_{Z, 1, \bar{q}} \lesssim\|u(\tau, \cdot)\|_{Z, 1,2}+\|\p u(\tau, \cdot)\|_{Z, 1,2} .
\end{equation}
Moreover, according to $p>2$, one has
$$
\begin{aligned}
\left\||u(\tau, \cdot)|^{p-1}\right\|_{2+\varepsilon_2}  \leq\|u(\tau, \cdot)\|_{\left(2+\varepsilon_2\right)(p-1)}^{p-1} &\lesssim\tau^{-\left(\frac{1}{2}-\frac{1}{\left(2+\varepsilon_2\right)(p-1)}\right)(p-1)}\|u(\tau, \cdot)\|_{Z, 1,2}^{p-1} \\
& \leq\|u(\tau, \cdot)\|_{Z, 1,2}^{p-1}.
\end{aligned}
$$
This, together with \eqref{L18} and \eqref{L17}, yields
$$
\begin{aligned}
\|\p \mathcal{N}u(t, \cdot)\|_{Z, 1, 2} &\lesssim t^{-\f{\mu}{2}}\ve\left\|\nabla u_0\right\|_{Z, 1, 2}
+t^{-\f{\mu}{2}}\ve\left\|u_1\right\|_{Z, 1, 2}+t^{-1}\ve\left\|u_0+u_1\right\|_{Z, 1,2}\\
&\quad+ t^{-1}\|u\|_{X(T)}^p\int_{1}^{t} \tau^{1+p(\delta-1)}\mathrm{d}\tau.
\end{aligned}
$$
For $p>2$, we choose a fixed small constant $\varepsilon_1>0$ such that $1+p(\delta(\varepsilon_1)-1)<-1$ holds.
Thus,
\begin{equation}\label{L19}
\begin{aligned}
t \|\p \mathcal{N} u(t, \cdot)\|_{Z, 1,2} &\lesssim \ve\left\|\nabla u_0\right\|_{Z, 1,2}+\ve\left\|u_1\right\|_{Z, 1,2}
+\ve\left\|u_0+u_1\right\|_{Z, 1,2} +\|u\|_{X(T)}^p\\
&  \lesssim \ve\left\|\left(u_{0}, u_{1}\right)\right\|_{\mathcal{A}}+ \|u\|_{X(T)}^p.
\end{aligned}
\end{equation}
Combining \eqref{L15} and \eqref{L19} derives \eqref{L3}.

We next prove the contractive estimate \eqref{L4}. As treated in \eqref{L8}, we have that for $u, v\in X(T)$ and $p>2$,
\begin{equation}\label{L20}
\begin{aligned}
\|(\mathcal{N} u-\mathcal{N} v)(t, \cdot)\|_{Z, 1, 2}& \lesssim t^{\delta-1}\int_{1}^{t} \tau\left\||u-v|(\tau, \cdot)(|u|+|v|)^{p-1}(\tau, \cdot)\right\|_{Z, 1,(1+\varepsilon_1, 2)}\mathrm{d}\tau\\
& \lesssim t^{\delta-1}\int_{1}^{t} \tau \left\|(|u|+|v|)^{p-1}(\tau, \cdot)\right\|_{(\tilde{q}, \infty)}\||u-v|(\tau, \cdot)\|_{Z, 1,2}\mathrm{d}\tau\\
&\lesssim t^{\delta-1} \|u-v\|_{X(T)}\big(\|u\|_{X(T)}^{p-1}+\|v\|_{X(T)}^{p-1}\big)\int_{1}^{t} \tau^{1+\frac{p-1-\kappa}{2}+p(\delta-1)}\mathrm{d}\tau\\
&\lesssim t^{\delta-1} \|u-v\|_{X(T)}\big(\|u\|_{X(T)}^{p-1}+\|v\|_{X(T)}^{p-1}\big),
\end{aligned}
\end{equation}
where $\tilde{q}$ is defined by \eqref{L14-1}. On the other hand, employing the same approach for deriving \eqref{L19}, one can obtain that for  $u, v\in X(T)$ and $p>2$,
\begin{equation}\label{L21}
\begin{aligned}
\|\p(\mathcal{N} u-\mathcal{N} v)(t, \cdot)\|_{Z, 1, 2}&
\lesssim t^{-1}\int_{1}^{t} \tau\left\||u-v|(\tau, \cdot)(|u|+|v|)^{p-1}(\tau, \cdot)\right\|_{Z, 1,2}\mathrm{d}\tau\\
& \lesssim t^{-1}\int_{1}^{t} \tau \left\|(|u|+|v|)^{p-1}(\tau, \cdot)\right\|_{_{2+\varepsilon_1}}\||u-v|(\tau, \cdot)\|_{Z, 1, \bar{q}}\mathrm{d}\tau\\
&\lesssim t^{-1} \|u-v\|_{X(T)}\big(\|u\|_{X(T)}^{p-1}+\|v\|_{X(T)}^{p-1}\big)\int_{1}^{t} \tau^{1+p(\delta-1)}\mathrm{d}\tau\\
&\lesssim t^{-1} \|u-v\|_{X(T)}\big(\|u\|_{X(T)}^{p-1}+\|v\|_{X(T)}^{p-1}\big),
\end{aligned}
\end{equation}
where $\bar{q}$ is defined by \eqref{L17-1}.  Therefore, \eqref{L4} is shown and then Proposition \ref{prop-1}
is established by collecting \eqref{L1}, \eqref{L20}-\eqref{L21} and \eqref{L3}.

\end{proof}

Based on Proposition \ref{prop-1}, we start to prove Theorem \ref{YH-1}.
\vskip 0.1 true cm

\begin{proof}
[Proof of Theorem~\ref{YH-1}]
 Choosing $M=3C\left\|\left(u_{0}, u_{1}\right)\right\|_{\mathcal{A}}$ in $X(T, M)$, where $C$ is the positive constant mentioned in
 Proposition \ref{prop-1}. Then for any $u \in X(T, M)$, one can obtain
\begin{equation}\label{L6}
\begin{aligned}
\|\mathcal{N} u\|_{X(T)}
& \leq C\ve\left\|\left(u_{0}, u_{1}\right)\right\|_{\mathcal{A}}+C\|u\|_{X(T)}^{p} \\
& \leq C \ve\left\|\left(u_{0}, u_{1}\right)\right\|_{\mathcal{A}}+C\left(3 C\ve\left\|\left(u_{0}, u_{1}\right)\right\|_{\mathcal{A}}\right)^{p} \\
& \leq\left(C+3^{p} C^{p+1}\left\|\left(u_{0}, u_{1}\right)\right\|_{\mathcal{A}}^{p-1}\ve^{p-1}\right)\ve\left\|\left(u_{0}, u_{1}\right)\right\|_{\mathcal{A}} \\
& \leq 3 C\ve\left\|\left(u_{0}, u_{1}\right)\right\|_{\mathcal{A}},
\end{aligned}
\end{equation}
where $\ve\le \big(\f{2}{3^pC^p}\big)^{\f{1}{p-1}}\f{1}{1+\left\|\left(u_{0}, u_{1}\right)\right\|_{\mathcal{A}}}$. Similarly,
\begin{equation}\label{L7}
\begin{aligned}
\|\mathcal{N} u-\mathcal{N} v\|_{X(T)} &\leq C\|u-v\|_{X(T)}\left(\|u\|_{X(T)}^{p-1}+\|v\|_{X(T)}^{p-1}\right)\\
&\leq C\|u-v\|_{X(T)}\cdot2\left(3 C\ve\left\|\left(u_{0}, u_{1}\right)\right\|_{\mathcal{A}}\right)^{p-1}\\
&\leq \f{1}{2}\|u-v\|_{X(T)},
\end{aligned}
\end{equation}
provided that $\ve\le\f{{(4C)}^{-\f{1}{p-1}}}{3C(1+\left\|\left(u_{0}, u_{1}\right)\right\|_{\mathcal{A}})}$. Therefore,
by setting $\varepsilon_0:=\min\{\big(\f{2}{3^pC^p}\big)^{\f{1}{p-1}}\f{1}{1+\left\|\left(u_{0}, u_{1}\right)\right\|_{\mathcal{A}}},$
$\f{{(4C)}^{-\f{1}{p-1}}}{3C(1+\left\|\left(u_{0}, u_{1}\right)\right\|_{\mathcal{A}})}\}$ and $\ve\leq \varepsilon_0$,
one can conclude that $\mathcal{N}$ is a contractive mapping from $X(T, M)$ into $X(T, M)$. Therefore, there exists a unique solution $u \in X(T, M)$ such that
$\mathcal{N}u=u$. This yields that $u$ is a global solution  of \eqref{equ:eff1} since the related constant $C>0$ is independent
of $T$. Then the proof of Theorem \ref{YH-1} is completed.
\end{proof}

\vskip 0.2 true cm

{\bf Acknowledgements}. Yin Huicheng wishes to express his deep gratitude to Professor Ingo Witt
(University of G\"ottingen, Germany) for his constant interests in this problem and many fruitful discussions in the past.
In addition, the authors would like to thank Dr. He Daoyin  for his invaluable discussions.

\vskip 0.2 true cm

{\bf \color{blue}{Conflict of Interest Statement:}}

\vskip 0.1 true cm

{\bf The authors declare that there is no conflict of interest in relation to this article.}

\vskip 0.2 true cm
{\bf \color{blue}{Data availability statement:}}

\vskip 0.1 true cm

{\bf  Data sharing is not applicable to this article as no data sets are generated
during the current study.}

\vskip 0.2 true cm


\end{document}